\documentclass[reqno]{amsart}
 
\usepackage{enumitem}
\usepackage{bbm}
\usepackage{nicefrac}
\usepackage{hyperref}

\newcommand{\mailto}[1]{\href{mailto:#1}{\nolinkurl{#1}}}

\newcommand{\msc}[1]{\href{http://www.ams.org/msc/msc2010.html?t=&s=#1}{#1}}

\newtheorem{theorem}{Theorem}[section]
\newtheorem{definition}[theorem]{Definition}
\newtheorem{lemma}[theorem]{Lemma}
\newtheorem{proposition}[theorem]{Proposition}
\newtheorem{corollary}[theorem]{Corollary}

\newtheorem{remark}[theorem]{Remark}

\newcommand{\R}{{\mathbb R}}
\newcommand{\N}{{\mathbb N}}
\newcommand{\Z}{{\mathbb Z}}
\newcommand{\C}{{\mathbb C}}

\newcommand{\spr}[2]{\langle #1 , #2 \rangle}

\newcommand{\gS}{\mathfrak{S}}

\newcommand{\Qr}{\mathsf{q}}
\newcommand{\Wr}{\mathsf{w}}

\newcommand{\Pe}{\mathrm{S}}
\newcommand{\T}{\mathrm{T}}
\newcommand{\rI}{\mathrm{I}}
\newcommand{\pr}{\mathrm{P}}
\newcommand{\Sr}{\mathrm{s}}
\newcommand{\KIO}{\mathrm{K}}
\newcommand{\JIO}{\mathrm{J}}

\newcommand{\E}{\mathrm{e}}

\newcommand{\supp}{\mathrm{supp}}
\newcommand{\tr}{\mathrm{tr}}

\newcommand{\re}{\mathrm{Re}}
\newcommand{\loc}{{\mathrm{loc}}}
\newcommand{\cc}{{\mathrm{c}}}

\newcommand{\dom}[1]{\mathrm{dom}(#1)}
\newcommand{\ran}[1]{\mathrm{ran}(#1)}
\renewcommand{\ker}[1]{\mathrm{ker}(#1)}
\newcommand{\mul}[1]{\mathrm{mul}(#1)}

\newcommand{\ledot}{\,\cdot\,}
\newcommand{\redot}{\cdot\,}

\newcommand{\NL}{(0-)}

\newcommand{\qd}{{[1]}}

\newcommand{\Hast}{\dot{H}^1[0,L)}
\newcommand{\Hasto}{\dot{H}^1_0[0,L)}
\newcommand{\Hastoinf}{\dot{H}^1_0[0,\infty)}
\newcommand{\cH}{\mathcal{H}}

\newcommand{\G}{\mathcal{G}}

\newcommand{\dip}{\upsilon}

\newcommand{\id}{{\mathbbm 1}}

\numberwithin{equation}{section}

\begin{document}

\title{Generalized indefinite strings with purely discrete spectrum}

\dedicatory{Dedicated to the memory of Sergey Nikolaevich Naboko (1950--2020)}

\author[J.\ Eckhardt]{Jonathan Eckhardt}
\address{Department of Mathematical Sciences\\ Loughborough University\\ Epinal Way\\ Loughborough\\ Leicestershire LE11 3TU \\ UK}
\email{\mailto{J.Eckhardt@lboro.ac.uk}}

\author[A.\ Kostenko]{Aleksey Kostenko}
\address{Faculty of Mathematics and Physics\\ University of Ljubljana\\ Jadranska ul.\ 21\\ 1000 Ljubljana\\ Slovenia\\ and 
Faculty of Mathematics\\ University of Vienna\\ Oskar-Morgenstern-Platz 1\\ 1090 Vienna\\ Austria}
\email{\mailto{Aleksey.Kostenko@fmf.uni-lj.si}}


\thanks{{\it Research supported by the Austrian Science Fund (FWF) under Grants No.\ P30715 and I-4600 (A.K.) as well as by the Slovenian Research Agency (ARRS) under Grant No.\ N1-0137 (A.K.)}}


\begin{abstract}
We establish criteria for the spectrum of a generalized indefinite string to be purely discrete and to satisfy Schatten--von Neumann properties. 
The results can be applied to the isospectral problem associated with the conservative Camassa--Holm flow and to Schr\"odinger operators with $\delta'$-interactions.
\end{abstract}

\maketitle
 
\section{Introduction}

In this article, we are concerned with the spectral problem
 \begin{align}\label{eqnDEIndStr}
  -f'' = z\, \omega f + z^2 \dip f
\end{align}
for a {\em generalized indefinite string} $(L,\omega,\dip)$. 
This means that $0< L\leq\infty$, $\omega$ is a real distribution in $H^{-1}_{\loc}[0,L)$ and $\dip$ is a non-negative Borel measure on $[0,L)$. 
Spectral problems of this form arise as Lax (isospectral) operators in the study of various nonlinear completely integrable systems  (most notably the Camassa--Holm equation~\cite{caho93}, where finite time blow-up makes it necessary to allow coefficients of low regularity in~\eqref{eqnDEIndStr}; see \cite{ConservCH}, \cite{ConservMP}). 
The particular case of~\eqref{eqnDEIndStr} when the measure $\dip$ vanishes identically and $\omega$ is a non-negative Borel measure on $[0,L)$ is known as a {\em Krein string} and has a venerable history~\cite{dymc76}, \cite{kakr74}, \cite{kowa82}. 
To the best of our knowledge, spectral problems of the form~\eqref{eqnDEIndStr} with non-trivial coefficients $\dip$ first appeared in the work of M.~G.~Krein and H.~Langer~\cite{krla79}, \cite{la76} in their study of indefinite analogues of the moment problem. 
In the above generality, this spectral problem was introduced quite recently in~\cite{IndefiniteString}. 
Similar to Krein strings, which serve as a canonical model for operators with non-negative simple spectrum (roughly speaking, an arbitrary self-adjoint operator in a separable Hilbert space with non-negative simple spectrum is unitarily equivalent to a Krein string), the spectral problem~\eqref{eqnDEIndStr} is another canonical model (two other famous models are Jacobi matrices and $2\times 2$ canonical systems) for self-adjoint operators with simple spectrum; see~\cite{IndefiniteString}. 
The purpose of this article is to address the question under which conditions on the coefficients the spectrum $\sigma$ of~\eqref{eqnDEIndStr} is purely discrete (that is, consists of isolated eigenvalues without finite accumulation points) or satisfies
\begin{align}\label{eqnTrP}
\sum_{\lambda\in\sigma}\frac{1}{|\lambda|^p} < \infty
\end{align}
for some positive constant $p$. 
Our main results provide a complete characterization of generalized indefinite strings that give rise to purely discrete spectrum or that satisfy~\eqref{eqnTrP} for a constant $p>1$. 

In the case of Krein strings, the corresponding results have been known since the late 1950s. 
The discreteness criterion was first established by I.~S.~Kac and M.~G.~Krein in \cite{kakr58}, which has been published only in Russian. 
Due to positivity, one can employ a variational reformulation, which allows to reduce the question about discreteness to the study of the embedding of a form domain into the initial Hilbert space (compare~\cite{maz}). 
In this context, the Kac--Krein criterion is related to the Muckenhoupt inequalities~\cite{muc}. 

Removing the positivity assumption makes the corresponding considerations much more complicated. 
For instance, despite its fundamental importance, a discreteness criterion for $2\times 2$ canonical systems was found only recently by R.~Romanov and H.~Woracek~\cite{rowo19} (see also~\cite{rs}). 
Surprisingly enough (at least to the authors), a discreteness criterion for {\em indefinite strings} (the case when the measure $\dip$ vanishes identically and $\omega$ is a real-valued Borel measure on $[0,L)$) has essentially been available since the 1970s, when C.\ A.\ Stuart \cite{st73} established a compactness criterion for integral operators in the Hilbert space $L^2[0,\infty)$ of the form
\begin{align}\label{eqnJ}
\JIO\colon f\mapsto  \int_0^\infty \Qr(\max(\ledot,t))f(t)dt, 
\end{align}
for a function $\Qr$ in $L^2_{\loc}[0,\infty)$. 
This is because the operator $\JIO$ is closely related (see Section~\ref{secIntoper}) to the resolvent of the indefinite string when $\Qr$ is an anti-derivate of $\omega$. 
Due to this connection, criteria for the spectrum $\sigma$ to satisfy~\eqref{eqnTrP} follow from criteria for the operator $\JIO$ to belong to the Schatten--von Neumann class $\gS_p$. 
Such criteria have been established in the work of A.~B.~Aleksandrov, S.~Janson, V.~V.~Peller, and R.~Rochberg~\cite{AJPR}. 
 This connection will play a crucial role in our approach to obtain discreteness criteria for generalized indefinite strings. 
 However, let us mention that the discreteness criteria as well as the criteria when the spectrum $\sigma$ of a generalized indefinite string satisfies~\eqref{eqnTrP} can also be obtained by using the recent results of~\cite{rowo19}, as there is a bijective correspondence between generalized indefinite strings and $2\times 2$ canonical systems; see~\cite[Section~6]{IndefiniteString}. 
On the other hand, one may look at our results as an alternative way to obtain discreteness criteria for $2\times 2$ canonical systems.

In conclusion, let us sketch the content of the article. 
Section~\ref{sec:prelim} is of preliminary character and collects necessary notions and facts on the spectral theory of generalized indefinite strings. 
In Section~\ref{secIntoper} we are concerned with the resolvent at zero energy of the spectral problem 
 \begin{align}
  -f'' = z \chi f 
\end{align}
 when $\chi$ is a distribution in $H^{-1}_{\loc}[0,L)$. 
This operator turns out to be closely related to an integral operator of the form~\eqref{eqnJ}, which allows us to translate several known results from~\cite{chev70}, \cite{st73}, \cite{AJPR}. 
In the following section, we will then introduce a quadratic operator pencil associated with a generalized indefinite string, which enables us to connect the spectral problem~\eqref{eqnDEIndStr} with the operators studied in Section~\ref{secIntoper}. 
This connection allows us to derive our main results  in Section~\ref{secDis}; a number of discreteness criteria for generalized indefinite strings. 
Consequently, in Section~\ref{secAPP} we apply our findings to the isospectral problem of the conservative Camassa--Holm flow
 \begin{align}
 -f'' + \frac{1}{4} f = z\, \omega f + z^2 \dip f.
\end{align}
Section~\ref{sec:Delta} provides another application of our results to one-dimensional Hamiltonians with $\delta'$-interactions.
Finally, in Appendix~\ref{app:IntOp} we gather a number of known results from~\cite{chev70}, \cite{st73}, \cite{AJPR} about integral operators of the form~\eqref{eqnJ} in a way that makes them convenient for us to apply. 
Throughout this article, we adopt the point of view of linear relations when dealing with linear operators. 
For the convenience of the reader, we summarize basic notions about linear relations in Appendix~\ref{app:LinRel}.
 
\subsection*{Notation}

 Let us first introduce several spaces of functions and distributions. 
 For $L\in(0,\infty]$, we denote with $L^2_{\loc}[0,L)$, $L^2[0,L)$ and $L^2_{\cc}[0,L)$ the spaces of locally square integrable functions, square integrable functions and square integrable functions with compact support in $[0,L)$, respectively. 
 The space $\dot{L}^2_{\cc}[0,L)$ consists of all functions $f$ in $L^2_{\cc}[0,L)$ with zero mean, that is, such that 
 \begin{align}
    \int_0^L f(x)dx = 0.
 \end{align}
 When $L$ is finite, the space $\dot{L}^2[0,L)$ can be defined in a similar way. 
 Furthermore, we denote with $H^1_{\loc}[0,L)$, $H^1[0,L)$ and $H^1_{\cc}[0,L)$ the usual Sobolev spaces 
\begin{align}
 H^1_{\loc}[0,L) & =  \lbrace f\in AC_{\loc}[0,L) \,|\, f'\in L^2_{\loc}[0,L) \rbrace, \\
 H^1[0,L) & = \lbrace f\in H^1_{\loc}[0,L) \,|\, f,\, f'\in L^2[0,L) \rbrace, \\ 
 H^1_{\cc}[0,L) & = \lbrace f\in H^1[0,L) \,|\, \supp(f) \text{ compact in } [0,L) \rbrace.
\end{align}
 The space of distributions $H^{-1}_{\loc}[0,L)$ is the topological dual space of $H^1_{\cc}[0,L)$. 
 We note that the mapping $\Qr\mapsto\chi$, defined by
 \begin{align}
    \chi(h) = - \int_0^L \Qr(x)h'(x)dx, \quad h\in H^1_{\cc}[0,L),
 \end{align} 
 establishes a one-to-one correspondence between $L^2_{\loc}[0,L)$ and $H^{-1}_{\loc}[0,L)$. 
 The unique function $\Qr\in L^2_{\loc}[0,L)$ corresponding to some distribution $\chi\in H^{-1}_{\loc}[0,L)$ in this way will be referred to as the {\em normalized anti-derivative} of $\chi$.
 Finally, a distribution in $H^{-1}_{\loc}[0,L)$ is said to be {\em real} if its normalized anti-derivative is real-valued almost everywhere on $[0,L)$. 

 A particular kind of distributions in $H^{-1}_{\loc}[0,L)$ arises from Borel measures on the interval $[0,L)$.
 More precisely, if $\chi$ is a non-negative Borel measure on $[0,L)$, then we will identify it with the distribution in $H^{-1}_{\loc}[0,L)$ given by  
 \begin{align}
  h \mapsto \int_{[0,L)} h\,d\chi. 
 \end{align}
 The normalized anti-derivative $\Qr$ of such a $\chi$ is simply given by the left-continuous distribution function 
 \begin{align}\label{eqnAntiChi}
 \Qr(x)=\int_{[0,x)}d\chi
 \end{align}
 for almost all $x\in [0,L)$, as an integration by parts (use, for example, \cite[Exercise~5.8.112]{bo07} or \cite[Theorem~21.67]{hest65}) shows.  

 In order to be able to introduce a self-adjoint realization of the differential equation~\eqref{eqnDEIndStr} in a suitable Hilbert space later, we also define the function space  
\begin{align}
\Hast & = \begin{cases} \lbrace f\in H^1_{\loc}[0,L) \,|\, f'\in L^2[0,L),~ \lim_{x\rightarrow L} f(x) = 0 \rbrace, & L<\infty, \\ \lbrace f\in H^1_{\loc}[0,L) \,|\, f'\in L^2[0,L) \rbrace, & L=\infty, \end{cases} 
\end{align}
 as well as its linear subspace  
\begin{align}
 \Hasto & = \lbrace f\in \Hast \,|\, f(0) = 0 \rbrace, 
\end{align}
 which turns into a Hilbert space when endowed with the scalar product 
 \begin{align}\label{eq:normti}
 \spr{f}{g}_{\Hasto} = \int_0^L f'(x) g'(x)^\ast dx, \quad f,\, g\in\Hasto,
 \end{align}
 where we use a star to denote complex conjugation. 
  The space $\Hasto$ can be viewed as a completion of the space of all smooth functions which have compact support in $(0,L)$ with respect to the norm induced by~\eqref{eq:normti}. 
 In particular, the space $\Hasto$ coincides algebraically and topologically with the usual Sobolev space $H^1_0[0,L)$ when $L$ is finite. 
 Functions in $\Hasto$ are not necessarily bounded, but they satisfy the simple growth estimate 
 \begin{align}\label{eqnFuncHastGrow}
  |f(x)|^2 \leq x\biggl(1-\frac{x}{L}\biggr)\, \|f\|^2_{\Hasto}, \quad x\in[0,L),~f\in\Hasto.
 \end{align}
 Here we employ the convention that whenever an $L$ appears in a denominator, the corresponding fraction has to be interpreted as zero if $L$ is not finite.

\section{Generalized indefinite strings}\label{sec:prelim}

A generalized indefinite string is a triple $(L,\omega,\dip)$ such that $L\in(0,\infty]$, $\omega$ is a real distribution in $H^{-1}_{\loc}[0,L)$ and $\dip$ is a non-negative Borel measure on $[0,L)$.  
 Associated with such a generalized indefinite string is the inhomogeneous differential equation
 \begin{align}\label{eqnDEinho}
  -f''  = z\, \omega f + z^2 \dip f + \chi, 
 \end{align}
 where $\chi$ is a distribution in $H^{-1}_{\loc}[0,L)$ and $z$ is a complex spectral parameter. 
 Of course, this differential equation has to be understood in a weak sense:   
  A solution of~\eqref{eqnDEinho} is a function $f\in H^1_{\loc}[0,L)$ such that 
 \begin{align}
  f'\NL h(0) + \int_{0}^L f'(x) h'(x) dx = z\, \omega(fh) + z^2 \int_{[0,L)}fh\,d\dip +  \chi(h)
 \end{align}
 for all $h\in H^1_{\cc}[0,L)$ and a (unique) constant $f'\NL\in\C$. 
 All differential equations in this article will be of the form~\eqref{eqnDEinho} and we only refer to \cite{IndefiniteString} for further details. 

 Each generalized indefinite string $(L,\omega,\dip)$ gives rise to a self-adjoint linear relation $\T$ in the Hilbert space
 \begin{align}
 \cH = \Hasto\times L^2([0,L);\dip),
\end{align}
 which is endowed with the scalar product
\begin{align}
 \spr{f}{g}_{\cH} = \int_0^L f_1'(x) g_1'(x)^\ast dx + \int_{[0,L)} f_2(x) g_2(x)^\ast d\dip(x), \quad f,\, g\in \cH.
\end{align} 
 Here we denote the respective components of some vector $f\in\cH$ by adding subscripts, that is, with $f_1$ and $f_2$.  
 
 \begin{definition}\label{defLRT}
The linear relation $\T$ in the Hilbert space $\cH$ is defined by saying that some pair $(f,g)\in\cH\times\cH$ belongs to $\T$ if and only if  
\begin{align}\label{eqnDEre1}
-f_1'' & =\omega g_{1} + \dip g_{2}, &  \dip f_2 & =\dip g_{1}.
\end{align}
\end{definition}

 In order to be precise, we point out that the right-hand side of the first equation in~\eqref{eqnDEre1} has to be understood as the $H_{\loc}^{-1}[0,L)$ distribution   
\begin{align}
  h\mapsto \omega(g_1h) + \int_{[0,L)} g_2 h\, d\dip
\end{align} 
 and that the second equation means that $f_2$ is equal to $g_1$ almost everywhere on $[0,L)$ with respect to the measure $\dip$. 
  The linear relation $\T$ defined in this way turns out to be self-adjoint in the Hilbert space $\cH$; see~\cite[Section~4]{IndefiniteString} for more details. 

  Later on, we will use the following description of the resolvent of $\T$ that can be found in~\cite[Proposition 4.3]{IndefiniteString}.
  To this end, we recall that the Wronski determinant $W(\psi,\phi)$ of two solutions $\psi$, $\phi$ of the homogeneous differential equation
   \begin{align}\label{eqnDEho}
  -f'' = z\, \omega f + z^2 \dip f
\end{align}
 is defined as the unique number such that 
  \begin{align}
    W(\psi,\phi) =   \psi(x)\phi'(x)-\psi'(x)\phi(x)  
  \end{align}
 for almost all $x\in[0,L)$. 
 It is known that the Wronskian $W(\psi,\phi)$ is non-zero if and only if the functions $\psi$ and $\phi$ are linearly independent; see \cite[Corollary~3.3]{IndefiniteString}.
 
\begin{proposition}\label{propRes}
  If $z$ belongs to the resolvent set of $\T$, then one has   
  \begin{align}\begin{split}
   z (\T-z)^{-1}  g(x) & = \spr{g}{\G(x,\redot)^\ast}_{\cH} \begin{pmatrix} 1 \\ z \end{pmatrix} - g_1(x) \begin{pmatrix} 1 \\ 0 \end{pmatrix}, \quad x\in[0,L),  
 \end{split}\end{align} 
 for every $g\in\cH$, where the Green's function $\G$ is given by  
 \begin{align}
  \G(x,t) =  \begin{pmatrix} 1 \\ z \end{pmatrix} \frac{1}{W(\psi,\phi)} \begin{cases} \psi(x) \phi(t), & t\in[0,x), \\ \psi(t) \phi(x), & t\in[x,L),    \end{cases}
 \end{align}
 and $\psi$, $\phi$ are linearly independent solutions of the homogeneous differential equation~\eqref{eqnDEho} such that $\phi$ vanishes at zero, $\psi$ lies in $\dot{H}^1[0,L)$ and $z\psi$ lies in $L^2([0,L);\dip)$. 
 \end{proposition}    
  
 For the sake of simplicity, we shall always mean the spectrum of the corresponding linear relation $\T$ when we refer to the spectrum $\sigma$ of a generalized indefinite string $(L,\omega,\dip)$. 
 The same convention also applies to the various spectral types. 
 In particular, we say that the spectrum $\sigma$ of a generalized indefinite string $(L,\omega,\dip)$ is purely discrete if the linear relation $\T$ has purely discrete spectrum.

\section{Some integral operators in \texorpdfstring{$\Hasto$}{H01[0,L)}}\label{secIntoper}

 Throughout this section, let $\chi$ be a distribution in $H^{-1}_{\loc}[0,L)$ and denote with $\Qr$ its normalized anti-derivative.
 We introduce the linear relation $\KIO_\chi$ in $\Hasto$ by defining that a pair $(g,f)\in\Hasto\times\Hasto$ belongs to $\KIO_\chi$ if and only if  
  \begin{align}\label{eqnKchi}
   -f'' = \chi g,
  \end{align}
 where equality again has to be understood in a distributional sense.  

 \begin{proposition}\label{propKIO}
   The linear relation $\KIO_\chi$ is (the graph of) a densely defined closed linear operator with core $\Hasto\cap H_{\cc}^1[0,L)$ such that 
   \begin{align}\label{eqnKIOchiIO}
     \KIO_\chi g(x) =  \chi(\delta_x g), \quad x\in[0,L), 
   \end{align}
   for all $g\in\Hasto\cap H_{\cc}^1[0,L)$, where the kernel function $\delta_x$  is defined by 
  \begin{align}\label{eq:Kernel}
   \delta_x(t) = \min(x,t) \biggl(1-\frac{\max(x,t)}{L}\biggr), \quad x,\,t\in[0,L).
  \end{align}
 \end{proposition}

\begin{proof}
  We define the linear operator $\KIO_{\chi,0}$ with domain $\Hasto\cap H_{\cc}^1[0,L)$ by
  \begin{align*}
    \KIO_{\chi,0} h(x)= \chi(\delta_x h) = -\int_0^L \Qr(t)(\delta_xh)'(t)dt, \quad x\in[0,L),
  \end{align*}
 for $h\in\Hasto\cap H_{\cc}^1[0,L)$.
  In order to verify that the function $\KIO_{\chi,0}h$ indeed belongs to $\Hasto$, one first computes that 
 \begin{align*}
   \chi(\delta_x h) & =  \int_0^x \int_0^t \Qr(s) h'(s)ds\, dt  - \int_0^x \Qr(t)h(t)dt  \\
     & \qquad\qquad - x\int_0^L \Qr(t)h'(t)dt + \frac{x}{L} \int_0^L \Qr(t)(th'(t)+h(t))dt
 \end{align*}
 for all $x\in [0,L)$. 
 This shows that the function $x\mapsto\chi(\delta_x h)$ is locally absolutely continuous on $[0,L)$ with square integrable derivative given by 
 \begin{align}\label{eqnKIOchider}
      - \Qr(x)h(x) - \int_x^L \Qr(t)h'(t)dt + \frac{1}{L} \int_0^L \Qr(t)(th'(t)+h(t))dt
 \end{align}
 for almost every $x\in [0,L)$.
 Since it also shows that $\chi(\delta_x h)\rightarrow 0$ as $x\rightarrow L$ when $L$ is finite and that $\chi(\delta_0 h) = 0$, we conclude that $\KIO_{\chi,0}h$ belongs to $\Hasto$.
 By using the expression in~\eqref{eqnKIOchider} for the derivative of $\KIO_{\chi,0}h$, one verifies that 
 \begin{align*}
   -(\KIO_{\chi,0}h)'' = \chi h
 \end{align*}
  in a distributional sense, which implies that (the graph of) $\KIO_{\chi,0}$ is contained in $\KIO_{\chi}$. 
 Since the operator $\KIO_{\chi,0}$ is densely defined, its adjoint $\KIO_{\chi,0}^\ast$ is a closed linear operator in $\Hasto$.
 Hence, for the remaining claims it suffices to prove that 
 \begin{align*}
    \KIO_{\chi,0}^\ast \subseteq \KIO_{\chi^\ast} \subseteq \KIO_\chi^\ast \subseteq \KIO_{\chi,0}^\ast,
 \end{align*}
 where the last inclusion is evident as $\KIO_{\chi,0}\subseteq\KIO_\chi$.
 In order to verify the first inclusion, suppose that $(g,f)\in\KIO_{\chi,0}^\ast$ and let $h\in\Hasto\cap H_{\cc}^1[0,L)$.
 From the expression for the derivative of $\KIO_{\chi,0} h$ in~\eqref{eqnKIOchider} and an integration by parts, we get
 \begin{align*}
   \int_0^L f'(t)h'(t)^\ast dt  = \spr{f}{h}_{\Hasto} & = \spr{g}{\KIO_{\chi,0} h}_{\Hasto} \\
   & = \lim_{x\rightarrow L} \int_0^x g'(t)(\KIO_{\chi,0}h)'(t)^\ast dt \\
   & = \lim_{x\rightarrow L} - \int_0^x \Qr(t)^\ast (gh^\ast)'(t)dt
    = \chi^\ast(gh^\ast). 
 \end{align*}
 As it does not matter that all our test functions $h$ vanish at zero, this shows that the pair $(g,f)$ belongs to $\KIO_{\chi^\ast}$. 
 For the second inclusion, we need to prove that 
 \begin{align*}
   \spr{f}{g_\ast}_{\Hasto} = \spr{g}{f_\ast}_{\Hasto}
 \end{align*}
 when $(g_\ast,f_\ast)\in\KIO_{\chi^\ast}$ and $(g,f)\in\KIO_{\chi}$. 
 To this end, we first note that the respective differential equations that the pairs $(g_\ast,f_\ast)$ and $(g,f)$ satisfy entail that 
 \begin{align*}
   f_\ast'(x) + \Qr(x)^\ast g_\ast(x) & = d_{\ast} + \int_0^x \Qr(t)^\ast g_\ast'(t)dt =: f^{\qd}_\ast(x), \\
   f'(x) + \Qr(x) g(x) & = d + \int_0^x \Qr(t) g'(t)dt =: f^{\qd}(x),
 \end{align*}
 for almost all $x\in[0,L)$ and some constants $d$, $d_\ast\in\C$; see \cite[Equation~(3.6)]{IndefiniteString}. 
 Integration by parts then gives
\begin{align*}
  \int_0^x f'(t) g_\ast'(t)^\ast dt  - \int_0^x g'(t)f_\ast'(t)^\ast dt = \bigl. f^{\qd}(t) g_\ast(t)^\ast - g(t) f^{\qd}_\ast(t)^\ast \bigr|_{t=0}^x 
\end{align*}
 for every $x\in[0,L)$. 
  This clearly implies that 
\begin{align*}
  \spr{f}{g_\ast}_{\Hasto} - \spr{g}{f_\ast}_{\Hasto} = \lim_{x\rightarrow L} f^{\qd}(x) g_\ast(x)^\ast - g(x) f^{\qd}_\ast(x)^\ast
\end{align*}
and we are left to verify that the limit (which is already known to exist) is zero. 
However, this follows from the fact that the function 
\begin{align*}
  \frac{|f^{\qd}(x) g_\ast(x)^\ast - g(x) f^{\qd}_\ast(x)^\ast|^2}{x\bigl(1-\frac{x}{L}\bigr)} =  \frac{|f'(x) g_\ast(x)^\ast - g(x) f'_\ast(x)^\ast|^2}{x\bigl(1-\frac{x}{L}\bigr)}
\end{align*}
is integrable near $L$ due to the estimate in~\eqref{eqnFuncHastGrow} applied to $g_\ast$ and $g$. 
\end{proof}
 
  In the course of the proof of Proposition~\ref{propKIO}, we found the adjoint of $\KIO_\chi$. 
 
  \begin{corollary}\label{corKIOadj}
    The adjoint of the operator $\KIO_\chi$ is given by 
    \begin{align}
      \KIO_\chi^\ast = \KIO_{\chi^\ast}.
    \end{align}
    In particular, the operator $\KIO_\chi$ is self-adjoint when the distribution $\chi$ is real.
  \end{corollary}

  An inspection of the definition of the operator $\KIO_\chi$ also proves the following. 

   \begin{corollary}\label{corKIOform}
    For all functions $g$, $h\in\Hasto\cap H_{\cc}^1[0,L)$ one has  
    \begin{align}
      \spr{\KIO_\chi g}{h}_{\Hasto} = \chi(gh^\ast).
    \end{align}
  \end{corollary}

 The next result about boundedness of the operator $\KIO_\chi$ when $\chi$ is a non-negative Borel measure on $[0,L)$ will be useful in Section~\ref{secPencil}.

  \begin{corollary}\label{corKIOincl}
   Suppose that $\chi$ is a non-negative Borel measure on $[0,L)$. 
   The operator $\KIO_\chi$ is bounded if and only if the inclusion $\rI_\chi\colon \Hasto\rightarrow L^2([0,L);\chi)$ is bounded.
   In this case, the adjoint of $\rI_\chi$ is given by 
   \begin{align}
     \rI_\chi^\ast g(x) = \int_{[0,L)} \delta_x g\, d\chi, \quad x\in [0,L), 
   \end{align}
   for all functions $g\in L^2([0,L);\chi)$ and one has $\KIO_\chi = \rI_\chi^\ast\rI_\chi$.  
  \end{corollary}
  
  \begin{proof}
   For every $h\in\Hasto\cap H_{\cc}^1[0,L)$, we get from Corollary~\ref{corKIOform} that 
   \begin{align*}
     \spr{\rI_\chi h}{\rI_\chi h}_{L^2([0,L);\chi)} = \int_{[0,L)} |h|^2 d\chi = \chi(hh^\ast) = \spr{\KIO_\chi h}{h}_{\Hasto}.
   \end{align*}
   Since $\KIO_\chi$ is self-adjoint, it follows that the operator $\KIO_\chi$ is bounded if and only if the inclusion $\rI_\chi$ is bounded.
   In this case, the above equality also shows that $\KIO_\chi=\rI_\chi^\ast\rI_\chi$ and hence it remains to note that the adjoint of $\rI_\chi$ is given by
   \begin{align*}
    \rI_\chi^\ast g(x) = \spr{\rI_\chi^\ast g}{\delta_x}_{\Hasto} = \spr{g}{\rI_\chi \delta_x}_{L^2([0,L);\chi)} = \int_{[0,L)} \delta_x g\,d\chi, \quad x\in [0,L),
  \end{align*}
  for all functions $g\in L^2([0,L);\chi)$.     
   \end{proof}

The operator $\KIO_\chi$ turns out to be unitarily equivalent to a particular integral operator in $L^2[0,L)$, which allows us to readily translate the boundedness and compactness criteria collected in Appendix~\ref{app:IntOp}. 
For simplicity, we will state and prove the cases of an unbounded interval and a bounded interval separately. 

\begin{theorem}\label{thmKchiBC}
Suppose that $L$ is not finite. 
The following assertions hold true:
\begin{enumerate}[label=(\roman*), ref=(\roman*), leftmargin=*, widest=iii]
\item The operator $\KIO_\chi$ is bounded if and only if  there is a constant $c\in\C$ such that 
\begin{align}
  \limsup_{x\rightarrow\infty}\,  x\int_x^\infty |\Qr(t) - c|^2dt < \infty.
\end{align}
In this case, the constant $c$ is given by 
\begin{align}\label{eqnCqr}
c = \lim_{x\rightarrow\infty} \frac{1}{x}\int_0^{x} \Qr(t)dt.
\end{align} 
\item The operator $\KIO_\chi$ is compact if and only if there is a constant $c\in\C$ such that 
\begin{align}
  \lim_{x\rightarrow\infty}  x\int_x^\infty |\Qr(t) - c|^2dt = 0.
\end{align}
\item For each $p>1$, the operator $\KIO_\chi$ belongs to the Schatten--von Neumann class $\gS_p$ if and only if there is a constant $c\in\C$ such that 
\begin{align}
  \int_0^\infty \biggl(x\int_x^\infty |\Qr(t) - c|^2dt\biggr)^{\nicefrac{p}{2}} \frac{dx}{x} < \infty.
\end{align}
\item  If the operator $\KIO_\chi$ belongs to the Hilbert--Schmidt class $\gS_2$, then its Hilbert--Schmidt norm is given by
 \begin{align}
   \|\KIO_\chi\|_{\gS_2}^2 = 2 \int_0^\infty x|\Qr(x)-c|^2 dx,
 \end{align}
 where the constant $c$ is given by~\eqref{eqnCqr}.
 \item     If the operator $\KIO_\chi$ belongs to the trace class $\gS_1$, then 
    \begin{align}
      \int_0^\infty \biggl(x\int_x^\infty |\Qr(t) - c|^2dt\biggr)^{\nicefrac{1}{2}} \frac{dx}{x} < \infty,
    \end{align}
     the function $\Qr-c$ is integrable and the trace of $\KIO_\chi$ is given by 
    \begin{align}
      \tr\,\KIO_\chi = \int_0^\infty (c-\Qr(x))dx, 
    \end{align}
    where the constant $c$ is given by~\eqref{eqnCqr}.
\end{enumerate}
\end{theorem}

\begin{proof}
We first note that the map $\mathrm{U}:f\mapsto f'$ is unitary from $\Hastoinf$ to $L^2[0,\infty)$ with inverse simply given by
\begin{align*}
  \mathrm{U}^{-1} f(x) = \int_0^x f(t)dt,\quad x\in [0,\infty).
\end{align*}
If $f$ belongs to $\dot{L}^2_{\cc}[0,\infty)$, then $\mathrm{U}^{-1}f$ belongs to $\Hastoinf\cap H_{\cc}^1[0,\infty)$ and from the expression in~\eqref{eqnKIOchider} for the derivative of $\KIO_\chi h$ when $h\in\Hastoinf\cap H_{\cc}^1[0,\infty)$ we get 
\begin{align*}
   \mathrm{U}\KIO_\chi \mathrm{U}^{-1}f(x) = (\KIO_\chi \mathrm{U}^{-1}f)'(x)  = -\Qr(x) \int_0^x f(t) dt - \int_x^\infty \Qr(t)f(t)dt = -\JIO f(x)
\end{align*}
for almost every $x\in [0,\infty)$, where $\JIO$ is the integral operator in $L^2[0,\infty)$ defined in Appendix~\ref{app:IntOp}. 
 Now the claims follow from Theorem~\ref{thmAJPR}.
 \end{proof}

With the connection established in the proof of Theorem~\ref{thmKchiBC}, the criteria from Theorem~\ref{thmAJPR+} become readily available as well. 

\begin{theorem}\label{thmKchiBC+}
Suppose that $L$ is not finite and that $\chi$ is a non-negative Borel measure on $[0,\infty)$. 
The following assertions hold true:
\begin{enumerate}[label=(\roman*), ref=(\roman*), leftmargin=*, widest=iii]
\item\label{itmKchiBC+1} The operator $\KIO_\chi$ is bounded if and only if 
\begin{align}
\limsup_{x\rightarrow\infty}\, x \int_{[x,\infty)}d\chi <\infty.
\end{align}
\item The operator $\KIO_\chi$ is compact if and only if 
\begin{align}
\lim_{x\rightarrow \infty} x \int_{[x,\infty)}d\chi  =0.
\end{align}
\item For each $p>\nicefrac{1}{2}$, the operator $\KIO_\chi$ belongs to the Schatten--von Neumann class $\gS_p$ if and only if 
\begin{align}
\int_0^\infty  \biggl(x\int_{[x,\infty)}d\chi\biggr)^p\frac{dx}{x} <\infty. 
\end{align}
 \item If the operator $\KIO_\chi$ belongs to the trace class $\gS_1$, then its trace is given by
    \begin{align}
      \tr\,\KIO_\chi = \int_{[0,\infty)} x\, d\chi(x). 
    \end{align}
\item If the operator $\KIO_\chi$ belongs to the Schatten--von Neumann class $\gS_{\nicefrac{1}{2}}$, then the measure $\chi$ is singular with respect to the Lebesgue measure.
\end{enumerate}
\end{theorem}

In a similar way, it is also possible to obtain boundedness and compactness criteria for the operator $\KIO_\chi$ when the interval $[0,L)$ is bounded. 

\begin{theorem}\label{thmKchicompactfin}
Suppose that $L$ is finite. 
The following assertions hold true:
\begin{enumerate}[label=(\roman*), ref=(\roman*), leftmargin=*, widest=iii]
\item The operator $\KIO_\chi$ is bounded if and only if 
\begin{align}
  \limsup_{x\rightarrow L}\, (L-x)\int_0^x|\Qr(t)|^2dt  < \infty.
\end{align}
\item The operator $\KIO_\chi$ is compact if and only if 
\begin{align}
  \lim_{x\rightarrow L}  (L-x)\int_0^x|\Qr(t)|^2dt = 0.
\end{align}
\item For each $p>1$, the operator $\KIO_\chi$ belongs to the Schatten--von Neumann class $\gS_p$ if and only if 
\begin{align}
  \int_0^L \biggl((L-x)\int_0^x|\Qr(t)|^2dt\biggr)^{\nicefrac{p}{2}} \frac{dx}{L-x} < \infty.
\end{align}
\item  If the operator $\KIO_\chi$ belongs to the Hilbert--Schmidt class $\gS_2$, then its Hilbert--Schmidt norm is bounded by 
 \begin{align}
   \|\KIO_\chi\|_{\gS_2}^2 \leq 2\int_0^L (L-x)|\Qr(x)|^2 dx.
 \end{align}
\item\label{itmTFfinL} If the operator $\KIO_\chi$ belongs to the trace class $\gS_1$, then   
\begin{align}
   \int_0^L \biggl((L-x)\int_0^x |\Qr(t)|^2dt\biggr)^{\nicefrac{1}{2}}\frac{dx}{L-x} < \infty,
\end{align}
 the function $\Qr$ is integrable and the trace of $\KIO_\chi$ is given by  
\begin{align}\label{eqnTRchi}
\tr\,\KIO_\chi =  \int_0^L \biggl(\frac{2x}{L} - 1\biggr)\Qr(x)dx.
\end{align}
\end{enumerate}
\end{theorem}

\begin{proof}
 We first note that the map $\mathrm{U}:f\mapsto f'$ is unitary from $\Hasto$ to $\dot{L}^2[0,L)$ with inverse simply given by
\begin{align*}
  \mathrm{U}^{-1} f(x) = - \int_x^L f(t)dt,\quad x\in [0,L).
\end{align*}
If $f$ belongs to $\dot{L}^2_{\cc}[0,L)$, then $\mathrm{U}^{-1}f$ belongs to $\Hasto\cap H_{\cc}^1[0,L)$ and from the expression in~\eqref{eqnKIOchider} for the derivative of $\KIO_\chi h$ when $h\in\Hasto\cap H_{\cc}^1[0,L)$ we get  
\begin{align*}
   \mathrm{U}\KIO_\chi \mathrm{U}^{-1}f(x)  &  =   \int_x^L(\Qr(x)-\Qr(t))f(t)dt  + \frac{1}{L}\int_0^L \Qr(t)\biggl(t f(t) - \int_t^L f(s)ds\biggr)dt
\end{align*}
for almost every $x\in [0,L)$. 
The first term on the right-hand side becomes  
\begin{align*}
\int_x^L (\Qr(x) - \Qr(t))f(t)dt & = \Qr(x) \int_x^L f(t) dt + \int_0^x \Qr(t)f(t)dt - \int_0^L \Qr(t)f(t)dt.
\end{align*}
Furthermore, if $\JIO_L$ is the integral operator defined in~\eqref{eq:a07} with $\Qr_L=\Qr$, then an integration by parts gives 
\begin{align*}
\spr{\JIO_L f}{1}_{L^2[0,L)} & =  \int_0^L \biggl( \int_0^t \Qr(s) f(s) ds + \Qr(t)\int_t^L f(s)ds\biggr) dt \\
& =  -\int_0^L \Qr(t) \biggl(t f(t) - \int_t^L f(s) ds\biggr) dt + L\int_0^L \Qr(t)f(t)dt,
\end{align*}
 so that in combination we conclude that 
\begin{align}\label{eqnKIOJIO}
  \mathrm{U}\KIO_\chi \mathrm{U}^{-1}f =  \JIO_L f - \frac{1}{L} \spr{\JIO_L f}{1}_{L^2[0,L)} = \pr\JIO_L\pr f,
 \end{align}
  where $\pr\colon L^2[0,L)\rightarrow \dot{L}^2[0,L)$ is the orthogonal projection onto $\dot{L}^2[0,L)$. 
  This implies that the operator $\KIO_\chi$ is bounded if and only if the operator $\JIO_L$ is bounded. 
  In fact, boundedness of $\JIO_L$ clearly entails boundedness of $\KIO_\chi$. 
  On the other side, if the operator $\KIO_\chi$ is bounded, then we may conclude from~\eqref{eqnKIOJIO} that the linear functional $f\mapsto\spr{\JIO_L f}{1}_{L^2[0,L)}$ defined on the domain $\dot{L}^2_{\cc}[0,L)$ is closable because the integral operator $\JIO_L$ is closable (see \cite[Theorem~3.8]{hasu78} for example).
  Since this implies that the linear functional is bounded, we infer from~\eqref{eqnKIOJIO} that the operator $\JIO_L$ is bounded on $\dot{L}^2_{\cc}[0,L)$. 
  Finally, as the subspace $\dot{L}^2_{\cc}[0,L)$ has codimension one in the domain $L^2_{\cc}[0,L)$ of $\JIO_L$, we may conclude that the operator $\JIO_L$ is bounded. 
  Now the claims follow from Theorem~\ref{thmAJPR2}.
 In particular, one has  
\begin{align*}
\tr\,\KIO_\chi = \tr\,\JIO_L - \frac{1}{L}\spr{\JIO_L 1}{1}_{L^2[0,L)} = \int_0^L \Qr(x)dx - \frac{2}{L}\int_0^L (L-x)\Qr(x) dx,
\end{align*}   
 which gives the required trace formula in~\eqref{eqnTRchi}. 
\end{proof}

The connection established in the proof of Theorem~\ref{thmKchicompactfin} again also makes the boundedness and compactness criteria from Theorem~\ref{thmAJPR2+} available.

\begin{theorem}\label{thmKchiBC+fin1}
Suppose that $L$ is finite and that $\chi$ is a non-negative Borel measure on $[0,L)$. 
The following assertions hold true:
\begin{enumerate}[label=(\roman*), ref=(\roman*), leftmargin=*, widest=iii]
\item\label{itmKchiBC+fin1} The operator $\KIO_\chi$ is bounded if and only if 
\begin{align}
\limsup_{x\rightarrow L}\, (L-x)\int_{[0,x)} d\chi <\infty.
\end{align}
\item The operator $\KIO_\chi$ is compact if and only if 
\begin{align}
\lim_{x\rightarrow L} (L-x)\int_{[0,x)} d\chi  =0.
\end{align}
\item For each $p>\nicefrac{1}{2}$, the operator $\KIO_\chi$ belongs to the Schatten--von Neumann class $\gS_p$ if and only if 
\begin{align}
\int_0^L  \biggl((L-x)\int_{[0,x)} d\chi\biggr)^p\frac{dx}{L-x} <\infty. 
\end{align}
\item If the operator $\KIO_\chi$ belongs to the trace class $\gS_1$, then its trace is given by  
\begin{align}
\tr\,\KIO_\chi =  \int_{[0,L)} x\biggl(1-\frac{x}{L}\biggr)d\chi(x).
\end{align}
\item If the operator $\KIO_\chi$ belongs to the Schatten--von Neumann class $\gS_{\nicefrac{1}{2}}$, then the measure $\chi$ is singular with respect to the Lebesgue measure.
\end{enumerate}
\end{theorem}

\section{Quadratic operator pencils}\label{secPencil}
 
 We are now going to establish a connection between the integral operators from the previous section and a generalized indefinite string $(L,\omega,\dip)$. 
 To this end, we first introduce an associated quadratic operator pencil $\Pe$ in the Hilbert space $\Hasto$ as follows: 
  For every $z\in\C$, the linear relation $\Pe(z)$ in $\Hasto$ is defined by requiring that a pair $(f,g)\in\Hasto\times\Hasto$ belongs to $\Pe(z)$ if and only if  
  \begin{align}\label{eqnDEPeS}
   -f'' = z\,\omega f + z^2 \dip f - g''.
  \end{align}
  In order to be precise, we interpret $-g''$ here as the $H_{\loc}^{-1}[0,L)$ distribution 
  \begin{align}
   h \mapsto \int_0^L g'(x)h'(x)dx.
  \end{align} 

 \begin{proposition}\label{proppencil}
  For each $z\in\C$, the linear relation $\Pe(z)$ is (the graph of) a densely defined closed linear operator with core $\Hasto \cap H^1_{\cc}[0,L)$ and
   \begin{align}\label{eqnPeKIO}
  \Pe(z)  = \rI - \KIO_{z\omega+z^2\dip}. 
  \end{align}  
 \end{proposition}
 
 \begin{proof}
 By comparing the definition of the linear relation $\Pe(z)$ with the definition of the operator $\KIO_{z\omega+z^2\dip}$, we see that some pair $(f,g)\in\Hasto\times\Hasto$ belongs to $\Pe(z)$ if and only if $(f,f-g)$ belongs to $\KIO_{z\omega+z^2\dip}$, which shows~\eqref{eqnPeKIO}. 
 Now the claims follow readily from Proposition~\ref{propKIO}.
 \end{proof}
  
  We refer to $\Pe$ as a quadratic operator pencil because one can show that 
  \begin{align}
    \Pe(z) = \overline{\rI - z\KIO_\omega  - z^2 \KIO_\dip}, \quad z\in\C. 
  \end{align}
  From the relation~\eqref{eqnPeKIO} and Corollary~\ref{corKIOadj}, we are able to determine the adjoint.
  
  \begin{corollary}
    For each $z\in\C$, the adjoint of the operator $\Pe(z)$ is given by
    \begin{align}
     \Pe(z)^\ast = \Pe(z^\ast). 
    \end{align} 
    In particular, the operator $\Pe(z)$ is self-adjoint when $z$ is real. 
  \end{corollary}
  
  Since the operator pencil $\Pe$ and the linear relation $\T$ both arise from the same differential equation, it is not surprising that they are closely related.   
   
    \begin{proposition}\label{propSinv}
   If $z$ belongs to the resolvent set of $\T$, then the operator $\Pe(z)$ has an everywhere defined bounded inverse and  
   \begin{align} \label{eqnSinvT}
     \Pe(z)^{-1} =\pr \bigl( z(\T-z)^{-1} + \rI\bigr)\pr^\ast, 
   \end{align}
  where $\pr$ denotes the projection from $\cH$ onto $\Hasto$.
  \end{proposition} 
 
  \begin{proof}
  We are going to show first that for every $z\in\C$ one has   
   \begin{align}\label{eqnSinvTsub}
      \Pe(z)^{-1} \supseteq \{(g,f)\in\Pe(z)^{-1}\,|\, f\in L^2([0,L);\dip)\} \supseteq \pr \bigl( z(\T-z)^{-1} + \rI\bigr)\pr^\ast,
   \end{align}
   where the right-hand side should be understood as a product of linear relations. 
   In fact, if a pair $(g,f)\in\Hasto\times\Hasto$ belongs to $\pr \bigl( z(\T-z)^{-1} + \rI\bigr)\pr^\ast$, then there is an $h\in\cH$ with $Ph=f$ such that $(\pr^\ast g,h)$ belongs to $z(\T-z)^{-1} + \rI$. 
   This implies that $(h-\pr^\ast g,zh)$ belongs to $\T$ and hence 
  \begin{align}\label{eqnSinvDE}
 -(h_1-g)''  & = z\omega h_1 + z\dip h_2, & \dip h_2 & = z\dip h_1,
 \end{align}   
 by Definition~\ref{defLRT}, which shows that $(f,g)$ belongs to $\Pe(z)$. 
 If $z$ is not zero, then we conclude from the second equation in~\eqref{eqnSinvDE} that $f\in L^2([0,L);\dip)$. 
 Otherwise, when $z$ is zero, one notes that $\pr^\ast g = h$ belongs to the range of $\T$, which shows again that $f\in L^2([0,L);\dip)$ in view of Definition~\ref{defLRT}. 
 The converse inclusion
    \begin{align*}
      \{(g,f)\in\Pe(z)^{-1}\,|\, f\in L^2([0,L);\dip)\} \subseteq \pr \bigl( z(\T-z)^{-1} + \rI\bigr)\pr^\ast
   \end{align*}
 only holds for non-zero $z\in\C$ in general. 
 In order to prove it, we suppose that a pair $(f,g)$ belongs to $\Pe(z)$ such that $f\in L^2([0,L);\dip)$. 
 The definition of $\Pe(z)$ then shows that~\eqref{eqnSinvDE} holds with $h\in\cH$ given by $h_1=f$ and $h_2=zf$. 
 It follows that $(h-\pr^\ast g,zh)$ belongs to $\T$ and hence $(z\pr^\ast g,zh)$ belongs to $z(\T-z)^{-1} + \rI$, which shows that $(zg,zf)$ belongs to $\pr \bigl( z(\T-z)^{-1} + \rI\bigr)\pr^\ast$ and because $z$ is not zero, so does the pair $(g,f)$. 
   
  Finally, if $z$ belongs to the resolvent set of $\T$, then so does $z^\ast$ and we get
  \begin{align*}
    \Pe(z)^{-1} = (\Pe(z^\ast)^\ast)^{-1}  = (\Pe(z^\ast)^{-1})^\ast \subseteq \pr \bigl( z(\T-z)^{-1} + \rI\bigr)\pr^\ast
  \end{align*}
  from~\eqref{eqnSinvTsub}, which yields~\eqref{eqnSinvT} as well as the remaining claims.
 \end{proof}
   
 \begin{remark}  
  We have seen in the proof of Proposition~\ref{propSinv} that 
     \begin{align}\label{eqnSinvTeq}
      \Pe(z)^{-1} \supseteq \{(g,f)\in\Pe(z)^{-1}\,|\, f\in L^2([0,L);\dip)\} = \pr \bigl( z(\T-z)^{-1} + \rI\bigr)\pr^\ast
   \end{align}
  holds as long as $z$ is not zero. 
  The inclusion is indeed strict in some cases. 
 For example, if the measure $\dip$ is such that $\Hasto$ is not contained in $L^2([0,L);\dip)$ and we take $\omega = \dip$, then $\Pe(-1)$ is simply the identity operator, whereas the middle part in~\eqref{eqnSinvTeq} becomes its restriction to functions which belong to $L^2([0,L);\dip)$.
\end{remark}
   
   This connection with the linear relation $\T$ allows us to find a description of the inverse of $\Pe(z)$ via the resolvent of $\T$ when $z$ belongs to the resolvent set of $\T$.

   \begin{corollary}
  If $z$ belongs to the resolvent set of $\T$, then one has   
  \begin{align}
    \Pe(z)^{-1}g(x) & = \spr{g}{G(x,\redot)^\ast}_{\Hasto}, \quad x\in[0,L), 
  \end{align}
   for every $g\in \Hasto$,  where the Green's function $G$ is given by  
 \begin{align}
  G(x,t) = \frac{1}{W(\psi,\phi)} \begin{cases} \psi(x) \phi(t), & t\in[0,x), \\ \psi(t) \phi(x), & t\in[x,L),    \end{cases}
 \end{align}
 and $\psi$, $\phi$ are linearly independent solutions of the homogeneous differential equation~\eqref{eqnDEho} such that $\phi$ vanishes at zero, $\psi$ lies in $\dot{H}^1[0,L)$ and $z\psi$ lies in $L^2([0,L);\dip)$. 
 \end{corollary}
  
   \begin{proof}
    The claimed representation for the inverse of $\Pe(z)$ when $z$ belongs to the resolvent set of $\T$ follows immediately from Proposition~\ref{propSinv} and Proposition~\ref{propRes}. 
 \end{proof}  
  
  Under some additional assumptions on the coefficients, we can say much more about the relation between the operator pencil $\Pe$ and the linear relation $\T$. 

 \begin{theorem}\label{thmPencilpm}
  Zero belongs to the resolvent set of $\T$ if and only if the operators $\KIO_\omega$ and $\KIO_\upsilon$ are bounded.
  In this case, the following assertions hold true:
\begin{enumerate}[label=(\roman*), ref=(\roman*), leftmargin=*, widest=iii]  
  \item\label{itmPencilpm2} For all $z\in\C$ one has 
 \begin{align}
   \Pe(z) = \rI - z\KIO_{\omega} -z^2\KIO_{\dip}.
 \end{align}
  \item\label{itmPencilpm3} For all $z\in\C$ one has 
  \begin{align}
      \Pe(z)^{-1} = \pr \bigl( z(\T-z)^{-1} + \rI\bigr)\pr^\ast,
  \end{align}
    where $\pr$ denotes the projection from $\cH$ onto $\Hasto$.
  \item\label{itmPencilpm4} For each $z\in\C$, the operator $\Pe(z)$ has an everywhere defined bounded inverse if and only if $z$ belongs to the resolvent set of $\T$.   
   In this case, one has 
  \begin{align}\label{eqnTres}
   (\T-z)^{-1} =  \begin{pmatrix} \rI & 0 \\ z\rI_\dip & \rI \end{pmatrix} \begin{pmatrix}  \Pe(z)^{-1} & 0 \\ 0 & \rI \end{pmatrix} \begin{pmatrix} \rI  & z\rI_\dip^\ast  \\ 0 & \rI \end{pmatrix} \begin{pmatrix} \KIO_\omega & \rI_\dip^\ast  \\ \rI_\dip & 0 \end{pmatrix},
  \end{align}
   where the inclusion $\rI_\dip:\Hasto\rightarrow L^2([0,L);\dip)$ is bounded. 
     \item\label{itmPencilpm5} The inverse of $\T$ is given by
  \begin{align}\label{eqnTinvBOM}
   \T^{-1} = \begin{pmatrix} \KIO_\omega & \rI_\dip^\ast  \\ \rI_\dip & 0 \end{pmatrix}. 
  \end{align}
 \end{enumerate}
   \end{theorem}

 \begin{proof}
  Suppose first that zero belongs to the resolvent set of $\T$ and let $f=\T^{-1}g$, where $g\in\cH$ with $g_2=0$. 
  Definition~\ref{defLRT} implies that $f_2=\rI_\dip g_1\in L^2([0,L);\dip)$ and a comparison with the definition of $\KIO_\omega$ reveals that $f_1=\KIO_\omega g_1$. 
  From boundedness of $\T^{-1}$ we then infer that 
  \begin{align*}
    \|\KIO_\omega g_1\|_{\Hasto}^2 + \|\rI_\dip g_1\|_{L^2([0,L);\dip)}^2  = \|f\|_\cH^2 \leq \|\T^{-1}\|^2 \|g\|_\cH^2 = \|\T^{-1}\|^2 \|g_1\|_{\Hasto}^2.
  \end{align*}
  This shows that the operator $\KIO_\omega$ and the inclusion $\rI_\dip$ are bounded, which also guarantees boundedness of the operator $\KIO_\dip$ in view of Corollary~\ref{corKIOincl}. 
  On the other side, for $g\in\cH$ with $g_1=0$, Definition~\ref{defLRT} implies that $f=\T^{-1}g$ certainly satisfies $f_2=0$. 
  Because $\T^{-1}$ is self-adjoint, we conclude that it is given by~\eqref{eqnTinvBOM}.   
  Item~\ref{itmPencilpm2} follows readily from~\eqref{eqnPeKIO} and~\eqref{eqnKIOchiIO}, whereas Item~\ref{itmPencilpm3} follows from~\eqref{eqnSinvTeq} when $z$ is not zero and from~\eqref{eqnSinvT} when $z$ is zero (since zero belongs to the resolvent set of $\T$).
    For non-zero $z\in\C$, the equivalence in Item~\ref{itmPencilpm4} follows from the Frobenius--Schur factorization
  \begin{align*}
   \rI - z \T^{-1}  = \begin{pmatrix} \rI & -z\rI_\dip^\ast  \\ 0 & \rI \end{pmatrix}  \begin{pmatrix} \Pe(z) & 0  \\ 0 & \rI \end{pmatrix}  \begin{pmatrix} \rI & 0  \\ -z\rI_\dip & \rI \end{pmatrix},
  \end{align*}
  which also yields the identity in~\eqref{eqnTres} because
  \begin{align*}
     (\T-z)^{-1} = \bigl(\rI-z\T^{-1}\bigr)^{-1} \T^{-1}.
  \end{align*}
  It only remains to note that Item~\ref{itmPencilpm4} also holds when $z$ is zero.  
   
  In order to complete the proof, we need to show that boundedness of the operators $\KIO_\omega$ and $\KIO_\dip$ implies that zero belongs to the resolvent set of $\T$. 
  Under these assumptions, the inclusion $\rI_\dip$ is bounded by Corollary~\ref{corKIOincl} so that we may set 
  \begin{align*}
     f = \begin{pmatrix} \KIO_\omega & \rI_\dip^\ast  \\ \rI_\dip & 0 \end{pmatrix} g  = \begin{pmatrix} \KIO_\omega g_1 + \rI_\dip^\ast g_2 \\ \rI_\dip g_1 \end{pmatrix}
  \end{align*}
  for a given $g\in\cH$.   
  From the definition of $\KIO_\omega$ and after computing that 
    \begin{align*}
    \int_0^L (\rI_\dip^\ast g_2)'(x)h'(x)dx = \spr{\rI_\dip^\ast g_2}{h^\ast}_{\Hasto} = \spr{g_2}{\rI_\dip h^\ast}_{L^2([0,L);\dip)} = \int_{[0,L)} g_2 h\, d\dip
  \end{align*}
  for all functions $h\in\Hasto\cap H_{\cc}^1[0,L)$, we infer that  
  \begin{align*}
    -f_1'' & = \omega g_1 + \dip g_2, & \dip f_2 & = \dip g_1,
  \end{align*}
  which means that $(f,g)$ belongs to $\T$. 
  Since $g$ was arbitrary, this implies that the domain of the self-adjoint operator $\T^{-1}$ is the whole space and thus $\T^{-1}$ is bounded by the closed graph theorem, which guarantees that zero belongs to the resolvent set of $\T$. 
   \end{proof}  

 The operator pencil $\Pe$ is essentially a Schur complement of the block operator matrix in~\eqref{eqnTinvBOM} and the factorization~\eqref{eqnTres} is essentially a Frobenius--Schur decomposition.

 \begin{corollary}\label{corPelin}
  If zero belongs to the resolvent set of $\T$, then the non-zero spectrum of $\T^{-1}$ coincides with the non-zero spectrum of the block operator matrix\footnote{Here and below, $\sqrt{\KIO_\dip}$ always denotes the positive square root of the positive operator $\KIO_\dip$.}
 \begin{align}\label{eqnPeLinear}
   \begin{pmatrix} \KIO_\omega & \sqrt{\KIO_\dip}  \\ \sqrt{\KIO_\dip} & 0 \end{pmatrix}
 \end{align}
and all non-zero eigenvalues of the block operator matrix in~\eqref{eqnPeLinear} are simple. 
 \end{corollary}

 \begin{proof}
   The claim about the spectra follows from the Frobenius--Schur factorization
   \begin{align*}
       \rI - z\begin{pmatrix} \KIO_\omega & \sqrt{\KIO_\dip}  \\ \sqrt{\KIO_\dip} & 0 \end{pmatrix} = \begin{pmatrix} \rI & -z\sqrt{\KIO_\dip}  \\ 0 & \rI \end{pmatrix}  \begin{pmatrix} \Pe(z) & 0  \\ 0 & \rI \end{pmatrix}  \begin{pmatrix} \rI & 0  \\ -z\sqrt{\KIO_\dip} & \rI \end{pmatrix}
   \end{align*}
   and Theorem~\ref{thmPencilpm}~\ref{itmPencilpm4}.   
   If $\lambda$ is a non-zero eigenvalue of the block operator matrix in~\eqref{eqnPeLinear} with eigenvector $f\in\Hasto\times\Hasto$, then one readily finds that 
   \begin{align*}
      \Pe(\lambda^{-1})f_1 & =0, & f_2 & = \lambda^{-1}\sqrt{\KIO_\dip}f_1.
   \end{align*}
  However, the kernel of $\Pe(z)$ is at most one-dimensional as it consists of solutions of the homogeneous differential equation~\eqref{eqnDEho} that vanish at zero. 
  We conclude that the eigenvector $f$ is unique up to scalar multiples. 
 \end{proof}

 In view of the following section, let us point out that zero always belongs to the resolvent set of $\T$ when the spectrum of $\T$ is purely discrete. 

\begin{remark}\label{remTcomp}
If the resolvent $(\T-z)^{-1}$ is compact for some (and hence for all) $z$ in the resolvent set of $\T$, then zero belongs to the resolvent set of $\T$. 
Indeed,  the spectrum of $\T$ consists only of isolated eigenvalues in this case and since the kernel of $\T$ is trivial, it follows that zero can not be in the spectrum. 
Of course, compactness of the resolvent can be replaced by the weaker condition that zero does not belong to the essential spectrum of $\T$. 
\end{remark}

\section{Purely discrete spectrum}\label{secDis}

The main results of this article are a number of criteria for the spectrum $\sigma$ of a generalized indefinite string $(L,\omega,\dip)$ to be discrete and to satisfy 
\begin{align}\label{eqnSinSp}
\sum_{\lambda\in\sigma} \frac{1}{|\lambda|^p} < \infty
\end{align}
for a positive constant $p$.  
As a first step, we are going to relate these properties of the spectrum $\sigma$ to the corresponding operators $\KIO_\omega$ and $\KIO_\dip$. 

\begin{proposition}\label{propDSpecM}
 The following assertions hold true:
\begin{enumerate}[label=(\roman*), ref=(\roman*), leftmargin=*, widest=iii]
 \item\label{itmZeroResKIO} Zero does not belong to the spectrum $\sigma$ if and only if the operators $\KIO_\omega$ and $\KIO_\dip$ are bounded. 
 \item\label{itmDSpecKIO} The spectrum $\sigma$ is discrete if and only if the operators $\KIO_\omega$ and $\KIO_\dip$ are compact. 
 \item\label{itmSvNKIO} For each $p>0$, the spectrum $\sigma$ satisfies~\eqref{eqnSinSp} if and only if the operator $\KIO_\omega$ belongs to the Schatten--von Neumann class $\gS_{p}$ and the operator $\KIO_\dip$ belongs to the Schatten--von Neumann class $\gS_{\nicefrac{p}{2}}$.
 \item If the spectrum $\sigma$ satisfies~\eqref{eqnSinSp} with $p=2$, then   
\begin{align}
   \sum_{\lambda\in\sigma} \frac{1}{\lambda^2}  =  \|\KIO_\omega\|_{\gS_2}^2 + 2\,\tr\,\KIO_\dip. 
\end{align}
\item If the spectrum $\sigma$ satisfies~\eqref{eqnSinSp} with $p=1$, then 
\begin{align}
 \sum_{\lambda\in\sigma} \frac{1}{\lambda}  =  \tr\,\KIO_\omega. 
\end{align}
 \end{enumerate}
\end{proposition}

\begin{proof}
  Item~\ref{itmZeroResKIO} follows readily from Theorem~\ref{thmPencilpm} because the spectrum $\sigma$ of the generalized indefinite string $(L,\omega,\dip)$ is, by definition, the spectrum of the linear relation $\T$. 
  
  For Item~\ref{itmDSpecKIO}, suppose first that the spectrum $\sigma$ is discrete.
  Because the kernel of $\T$ is trivial, we infer that zero belongs to the resolvent set of $\T$.
  Corollary~\ref{corPelin} then implies that the self-adjoint block operator matrix in~\eqref{eqnPeLinear} is compact and thus the operators $\KIO_\omega$ and $\KIO_\dip$ are compact as well. 
  On the other side, if we suppose that $\KIO_\omega$ and $\KIO_\dip$ are compact, then zero belongs to the resolvent set of $\T$ in view of~\ref{itmZeroResKIO} and it follows from Corollary~\ref{corPelin} that the spectrum $\sigma$ is discrete because the block operator matrix in~\eqref{eqnPeLinear} is compact. 
  
 In order to prove Item~\ref{itmSvNKIO}, let us suppose first that the spectrum $\sigma$ satisfies~\eqref{eqnSinSp} so that the spectrum $\sigma$ is discrete and zero belongs to the resolvent set of $\T$.
 Corollary~\ref{corPelin} then implies that the self-adjoint block operator matrix in~\eqref{eqnPeLinear} belongs to the Schatten--von Neumann class $\gS_{p}$ and thus the operators $\KIO_\omega$ and $\sqrt{\KIO_\dip}$ belong to the Schatten--von Neumann class $\gS_{p}$ as well. 
 This clearly entails that the operator $\KIO_\dip$ belongs to the Schatten--von Neumann class $\gS_{\nicefrac{p}{2}}$.
  For the converse, we suppose that the operator $\KIO_\omega$ belongs to the Schatten--von Neumann class $\gS_{p}$ and that the operator $\KIO_\dip$ belongs to the Schatten--von Neumann class $\gS_{\nicefrac{p}{2}}$.
   The block operator matrix in~\eqref{eqnPeLinear} then belongs to the Schatten--von Neumann class $\gS_{p}$ because it is a sum  
  \begin{align*}
      \begin{pmatrix} \KIO_\omega & \sqrt{\KIO_\dip}  \\ \sqrt{\KIO_\dip} & 0 \end{pmatrix} = \begin{pmatrix} \KIO_\omega & 0 \\ 0 & 0 \end{pmatrix} + \begin{pmatrix} 0 & \sqrt{\KIO_\dip}  \\ \sqrt{\KIO_\dip} & 0 \end{pmatrix}
  \end{align*}
   of two block operator matrices that belong to the Schatten--von Neumann class $\gS_{p}$.
  It follows from Corollary~\ref{corPelin} that the spectrum $\sigma$ satisfies~\eqref{eqnSinSp}. 

 Finally, it remains to note that one has  
\begin{align*}
  \sum_{\lambda\in\sigma} \frac{1}{\lambda^2}  = \tr \begin{pmatrix} \KIO_\omega & \sqrt{\KIO_\dip}  \\ \sqrt{\KIO_\dip} & 0 \end{pmatrix}^2  = \tr \begin{pmatrix} \KIO_\omega^2 +\KIO_\dip & \KIO_\omega\sqrt{\KIO_\dip}  \\ \sqrt{\KIO_\dip}\KIO_\omega & \KIO_\dip \end{pmatrix} = \tr\,\KIO_\omega^2 + 2\,\tr\,\KIO_\dip
\end{align*} 
if the spectrum $\sigma$ satisfies~\eqref{eqnSinSp} with $p=2$ and  
\begin{align*}
 \sum_{\lambda\in\sigma} \frac{1}{\lambda} = \tr\begin{pmatrix} \KIO_\omega & \sqrt{\KIO_\dip}  \\ \sqrt{\KIO_\dip} & 0 \end{pmatrix} = \tr\,\KIO_\omega
\end{align*}
if the spectrum $\sigma$ satisfies~\eqref{eqnSinSp} with $p=1$.
\end{proof}

From Proposition~\ref{propDSpecM} and the corresponding boundedness and compactness criteria for the operators $\KIO_\omega$ and $\KIO_\dip$ in Section~\ref{secIntoper}, we readily derive the remaining theorems in this section with more explicit conditions in terms of the coefficients. 
The normalized anti-derivative of the distribution $\omega$ will be denoted with $\Wr$. 

\begin{theorem}\label{thmDSpec}
Suppose that $L$ is not finite. 
The following assertions hold true:
\begin{enumerate}[label=(\roman*), ref=(\roman*), leftmargin=*, widest=iii]
\item\label{itmZeroRes} Zero does not belong to the spectrum $\sigma$ if and only if  there is a constant $c\in\R$ such that 
\begin{align}
  \limsup_{x\rightarrow\infty}\,  x \int_x^\infty (\Wr(t) - c)^2dt  + x \int_{[x,\infty)} d\dip < \infty.
\end{align}
In this case, the constant $c$ is given by 
\begin{align}\label{eqnCqrm}
c = \lim_{x\rightarrow\infty} \frac{1}{x}\int_0^{x} \Wr(t)dt.
\end{align} 
\item The spectrum $\sigma$ is discrete  if and only if there is a constant $c\in\R$ such that 
\begin{align}
  \lim_{x\rightarrow\infty}  x\int_x^\infty (\Wr(t) - c)^2dt + x \int_{[x,\infty)} d\dip = 0.
\end{align}
\item For each $p>1$, the spectrum $\sigma$ satisfies~\eqref{eqnSinSp} if and only if there is a constant $c\in\R$ such that 
\begin{align}
  \int_0^\infty \biggl(x\int_x^\infty (\Wr(t) - c)^2dt + x \int_{[x,\infty)} d\dip\biggr)^{\nicefrac{p}{2}} \frac{dx}{x} < \infty.
\end{align}
\item If the spectrum $\sigma$ satisfies~\eqref{eqnSinSp} with $p=2$, then 
\begin{align}\label{eq:HSnorm1}
   \sum_{\lambda\in\sigma} \frac{1}{\lambda^2}  = 2 \int_0^\infty x (\Wr(x) - c)^2dx + 2\int_{[0,\infty)} x\,d\dip(x), 
\end{align}
where the constant $c$ is given by~\eqref{eqnCqrm}.
\item If the spectrum $\sigma$ satisfies~\eqref{eqnSinSp} with $p=1$, then 
\begin{align}\label{eqnDSpecTC}
  \int_0^\infty \biggl(x\int_x^\infty (\Wr(t) - c)^2dt \biggr)^{\nicefrac{1}{2}} \frac{dx}{x} < \infty, 
\end{align}
where the constant $c$ is given by~\eqref{eqnCqrm}, the function $\Wr - c$ is integrable with 
\begin{align}
 \sum_{\lambda\in\sigma} \frac{1}{\lambda}  =  \int_0^\infty (c-\Wr(x)) dx 
\end{align}
and the measure $\dip$ is singular with respect to the Lebesgue measure.
\end{enumerate}
\end{theorem}


One also obtains similar criteria for the case when $L$ is finite. 

\begin{theorem}\label{thmDSpecL}
Suppose that $L$ is finite. 
The following assertions hold true:
\begin{enumerate}[label=(\roman*), ref=(\roman*), leftmargin=*, widest=iii]
\item Zero does not belong to the spectrum $\sigma$ if and only if   
\begin{align}
  \limsup_{x\rightarrow L}\, (L-x)\int_0^x \Wr(t)^2dt  + (L-x)\int_{[0,x)} d\dip < \infty.
\end{align}
\item The spectrum $\sigma$ is discrete  if and only if 
\begin{align}
  \lim_{x\rightarrow L}  (L-x)\int_0^x \Wr(t)^2dt + (L-x)\int_{[0,x)} d\dip = 0.
\end{align}
\item For each $p>1$, the spectrum $\sigma$ satisfies~\eqref{eqnSinSp} if and only if 
\begin{align}
  \int_0^L \biggl((L-x)\int_0^x \Wr(t)^2dt + (L-x)\int_{[0,x)} d\dip \biggr)^{\nicefrac{p}{2}} \frac{dx}{L-x} < \infty.
\end{align}
\item If the spectrum $\sigma$ satisfies~\eqref{eqnSinSp} with $p=2$, then  
\begin{align}
   \sum_{\lambda\in\sigma} \frac{1}{\lambda^2}  \leq 2 \int_0^L (L-x)\Wr(x)^2 dx + 2 \int_{[0,L)} x \biggl(1-\frac{x}{L}\biggr)d\dip(x).
\end{align} 
\item If the spectrum $\sigma$ satisfies~\eqref{eqnSinSp} with $p=1$, then 
\begin{align}
   \int_0^L \biggl((L-x)\int_0^x \Wr(t)^2dt\biggr)^{\nicefrac{1}{2}}\frac{dx}{L-x} < \infty,
\end{align}
 the function $\Wr$ is integrable with 
\begin{align}
\sum_{\lambda\in \sigma}\frac{1}{\lambda} =  \int_0^L \biggl(\frac{2x}{L}-1\biggr)\Wr(x)dx
\end{align}
and the measure $\dip$ is singular with respect to the Lebesgue measure.
\end{enumerate}
\end{theorem}


The special cases considered in the remaining two theorems in this section are known as {\em Krein strings}  \cite{dymc76}, \cite{kakr74}, \cite{kowa82} in the literature.

\begin{theorem}\label{thmKacKrein}
Suppose that $L$ is not finite, that $\omega$ is a non-negative Borel measure on $[0,\infty)$ and that the measure $\dip$ vanishes identically.  
The following assertions hold true:
\begin{enumerate}[label=(\roman*), ref=(\roman*), leftmargin=*, widest=iii]
\item Zero does not belong to the spectrum $\sigma$ if and only if 
\begin{align}
 \limsup_{x\rightarrow\infty}\, x \int_{[x,\infty)}d\omega < \infty.
\end{align}
\item The spectrum $\sigma$ is discrete if and only if 
\begin{align}
 \lim_{x\rightarrow\infty} x \int_{[x,\infty)}d\omega =0.
\end{align}
\item For each $p>\nicefrac{1}{2}$, the spectrum $\sigma$ satisfies~\eqref{eqnSinSp} if and only if 
\begin{align}
   \int_0^\infty \biggl(x \int_{[x,\infty)}d\omega\biggr)^p \frac{dx}{x} <\infty.
\end{align}
\item If the spectrum $\sigma$ satisfies~\eqref{eqnSinSp} with $p=1$, then  
\begin{align}
 \sum_{\lambda\in\sigma} \frac{1}{\lambda}  =  \int_{[0,\infty)} x\,d\omega(x). 
\end{align}
\item If the spectrum $\sigma$ satisfies~\eqref{eqnSinSp} with $p=\nicefrac{1}{2}$, then the measure $\omega$ is singular with respect to the Lebesgue measure. 
\end{enumerate}
\end{theorem}

Again, we also get criteria for Krein strings when $L$ is finite. 

  \begin{theorem}\label{thmKacKreinL}
Suppose that $L$ is finite, that $\omega$ is a non-negative Borel measure on $[0,L)$ and that the measure $\dip$ vanishes identically.  
The following assertions hold true:
\begin{enumerate}[label=(\roman*), ref=(\roman*), leftmargin=*, widest=iii]
\item Zero does not belong to the spectrum $\sigma$ if and only if 
\begin{align}
 \limsup_{x\rightarrow L}\, (L-x)\int_{[0,x)} d\omega < \infty.
\end{align}
\item The spectrum $\sigma$ is discrete if and only if 
\begin{align}
 \lim_{x\rightarrow L} (L-x)\int_{[0,x)} d\omega =0. 
\end{align}
\item For each $p>\nicefrac{1}{2}$, the spectrum $\sigma$ satisfies~\eqref{eqnSinSp} if and only if 
\begin{align}
 \int_0^L \biggl( (L-x)\int_{[0,x)} d\omega\biggr)^p \frac{dx}{L-x} < \infty.
\end{align}
\item If the spectrum $\sigma$ satisfies~\eqref{eqnSinSp} with $p=1$, then  
\begin{align}
 \sum_{\lambda\in\sigma} \frac{1}{\lambda}  = \int_{[0,L)} x\biggl(1-\frac{x}{L}\biggr) d\omega(x). 
\end{align}
\item If the spectrum $\sigma$ satisfies~\eqref{eqnSinSp} with $p=\nicefrac{1}{2}$, then the measure $\omega$ is singular with respect to the Lebesgue measure. 
\end{enumerate}
\end{theorem}

\begin{remark}\label{rem:KacKrein} 
Items~(i) and~(ii) in Theorem~\ref{thmKacKrein} and Theorem~\ref{thmKacKreinL} were first proved by Kac and Krein in \cite{kakr58} (see also \cite{kakr74}). 
Item~(iii) in Theorem~\ref{thmKacKrein} and Theorem~\ref{thmKacKreinL} is due to Kac (see \cite[\S 3.2]{kac} for further details). 
The original approach of Kac and Krein is different from the one in \cite{AJPR} and it also provides quantitative bounds on the bottom of the essential spectrum. 
Positivity allows to employ variational techniques (for example, via embeddings of weighted $L^2$ and Sobolev spaces) in order to investigate discreteness of the spectrum and we refer to \cite[\S 1.3.1]{maz}, \cite{muc}  for further details.
\end{remark}

\begin{remark}
It was observed by Krein in the 1950s that for Krein strings
\begin{align}
\lim_{n\rightarrow \infty} \frac{ \#\{\lambda\in \sigma\,|\, \lambda<n^2\} }{n}= \frac{1}{\pi}\int_0^L \rho(x)dx,
\end{align}
where $\rho$ is the square root of the Radon--Nikod\'ym derivative of $\omega$ with respect to the Lebesgue measure. 
This, in particular, implies that the measure $\omega$ is singular with respect to the Lebesgue measure if \eqref{eqnSinSp} holds with $p=\nicefrac{1}{2}$. 
The study of eigenvalue distributions of strings with singular, or, more specifically, with self-similar coefficients, has attracted a considerable interest during the last decades and in this respect we only refer to \cite{sh15}, \cite{sove95}, \cite{vlsh06}, \cite{vlsh10} for further results.
\end{remark}

 \section{The isospectral problem for the conservative Camassa--Holm flow}\label{secAPP}

  In this section, we are going to demonstrate how our results apply to the isospectral problem of the conservative Camassa--Holm flow. 
  To this end, let $u$ be a real-valued function in $H^1_\loc[0,\infty)$ and $\dip$ be a non-negative Borel measure on $[0,\infty)$. 
  We define the distribution $\omega$ in $H^{-1}_\loc[0,\infty)$ by  
\begin{align}\label{eqnDefomega}
 \omega(h) = \int_0^\infty u(x)h(x)dx + \int_0^\infty u'(x)h'(x)dx, \quad h\in H^1_\cc[0,\infty),
\end{align}
so that $\omega = u - u''$ in a distributional sense. 
 The isospectral problem of the conservative Camassa--Holm flow is associated with the differential equation
 \begin{align}\label{eqnCHISP}
 -f'' + \frac{1}{4} f = z\, \omega f + z^2 \dip f,
\end{align}
where $z$ is a complex spectral parameter.
 Just like for generalized indefinite strings, this differential equation has to be understood in a weak sense in general (we refer to \cite{ConservCH}, \cite{CHPencil} and \cite[Section~7]{ACSpec} for further details).
 The differential equation~\eqref{eqnCHISP} gives rise to a self-adjoint linear relation $\T$ in the Hilbert space  
 \begin{align}
   \cH_0 = H_0^1[0,\infty)\times L^2([0,\infty);\dip)
 \end{align}
 equipped with the scalar product
 \begin{align}\begin{split}
   \spr{f}{g}_{\cH_0} & = \int_0^\infty f_1'(x)g_1'(x)^\ast dx + \frac{1}{4}\int_0^{\infty} f_1(x)g_1(x)^\ast dx \\
    & \qquad\qquad\qquad\qquad\qquad + \int_{[0,\infty)} f_2(x)g_2(x)^\ast d\dip(x), \quad f,\, g\in\cH_0,
 \end{split}\end{align}
 defined by saying that a pair $(f,g)\in\cH_0\times\cH_0$ belongs to $\T$ if and only if 
 \begin{align}
  -f_1'' + \frac{1}{4} f_1 & = \omega g_1 + \dip g_2, & \dip f_2 = \dip g_1.
 \end{align}
 More details can be found in \cite{bebrwe08}, \cite{LeftDefiniteSL} and \cite[Subsection~4.1]{CHPencil} in particular.
  
 We have shown in~\cite[Section~7]{ACSpec} that it is possible to transform the differential equation~\eqref{eqnCHISP} into the differential equation for a generalized indefinite string $(\infty,\tilde{\omega},\tilde{\dip})$ defined as follows:
 With the diffeomorphism $\Sr\colon[0,\infty)\rightarrow[0,\infty)$ given by  
 \begin{align}
  \Sr(t) = \log(1+t), \quad t\in[0,\infty),
 \end{align} 
 we define $\tilde{\Wr}$ to be a real-valued measurable function on $[0,\infty)$ such that   
 \begin{align}\label{eqnDefa}
   \tilde{\Wr}(t)  =  u(0) - \frac{u'(\Sr(t))+ u(\Sr(t))}{1+t}  
 \end{align}
 for almost all $t\in[0,\infty)$, where we note that the right-hand side is well-defined almost everywhere.
 The function $\tilde{\Wr}$ is locally square integrable, so that there is a real distribution $\tilde{\omega}$ in $H^{-1}_\loc[0,\infty)$ that has $\tilde{\Wr}$ as its normalized anti-derivative. 
 Furthermore, the non-negative Borel measure $\tilde{\dip}$ on $[0,\infty)$ is defined by setting 
 \begin{align}\label{eqnDefbeta}
   \tilde{\dip}(B) =  \int_B \frac{1}{1+t}\, d\dip\circ \Sr(t)  = \int_{\Sr(B)} \E^{-x} d\dip(x)
 \end{align}
 for all Borel sets $B\subseteq[0,\infty)$. 
 It then follows from \cite[Section~7]{ACSpec} that the spectrum of the linear relation $\T$ coincides with the spectrum of the generalized indefinite string $(\infty,\tilde{\omega},\tilde{\dip})$.
 From the criteria in Section~\ref{secDis}, we thus readily obtain similar criteria for the spectrum of $\T$ to be discrete and to satisfy 
\begin{align}\label{eqnSinSpCH}
\sum_{\lambda\in\sigma(\T)} \frac{1}{|\lambda|^p} < \infty
\end{align}
for some positive constant $p$.  
 
\begin{theorem}\label{thmCHdiscr}
The following assertions hold true:
\begin{enumerate}[label=(\roman*), ref=(\roman*), leftmargin=*, widest=iii]
\item Zero does not belong to the spectrum of $\T$ if and only if there is a constant $c\in\R$ such that  
 \begin{align}
   \limsup_{x\rightarrow\infty}\, \int_{x}^{\infty}\E^{x-t}\bigl(u'(t) + u(t) - c\, \E^{t}\bigr)^2dt + \int_{[x,\infty)}\E^{x-t}d\dip(t)  < \infty. 
 \end{align}
  In this case, the constant $c$ is given by 
 \begin{align}\label{eqnCHassumption}
  c =\lim_{x\rightarrow \infty} \E^{-x}\int_0^x u'(t) + u(t)\,dt.
\end{align}
\item The spectrum of $\T$ is discrete  if and only if there is a constant $c\in\R$ such that
 \begin{align}
   \lim_{x\rightarrow\infty} \int_{x}^{\infty}\E^{x-t}\bigl(u'(t) + u(t) - c\, \E^{t}\bigr)^2dt + \int_{[x,\infty)}\E^{x-t}d\dip(t) = 0. 
 \end{align}
\item For each $p>1$, the spectrum of $\T$ satisfies~\eqref{eqnSinSpCH} if and only if there is a constant $c\in\R$ such that 
\begin{align}\label{eqnCHp>1}
   \int_0^\infty \biggl(\int_{x}^{\infty}\E^{x-t}\bigl(u'(t) + u(t) - c\, \E^{t}\bigr)^2dt + \int_{[x,\infty)}\E^{x-t}d\dip(t)\biggr)^{\nicefrac{p}{2}} dx < \infty.
\end{align}
\item If the spectrum of $\T$ satisfies~\eqref{eqnSinSpCH} with $p=1$, then 
\begin{align}
 \int_0^\infty \biggl(\int_{x}^{\infty}\E^{x-t}\bigl(u'(t) + u(t) - c\, \E^{t}\bigr)^2dt\biggr)^{\nicefrac{1}{2}} dx < \infty, 
\end{align}
where the constant $c$ is given by~\eqref{eqnCHassumption}, the function $u' + u - c\E^{x}$ is integrable with 
\begin{align}
 \sum_{\lambda\in\sigma} \frac{1}{\lambda}  =  \int_0^\infty (c\, \E^x - u'(x) - u(x)) dx
\end{align}
and the measure $\dip$ is singular with respect to the Lebesgue measure.
\end{enumerate}
\end{theorem}

\begin{proof}
 Since the spectrum of $\T$ coincides with the spectrum of the generalized indefinite string $(\infty,\tilde{\omega},\tilde{\dip})$, the claims follow from Theorem~\ref{thmDSpec} after noting that 
 \begin{align*}
   \frac{1}{x}\int_0^{x} \tilde{\Wr}(t) dt & = u(0) -  \frac{1}{x}\int_{0}^{\log(1+x)} u'(t)+u(t)\, dt
\end{align*}
for all $x\in(0,\infty)$, as well as that 
\begin{align*}
  \int_x^y  (\tilde{\Wr}(t)-\tilde{c})^2 dt & = \int_{\log(1+x)}^{\log(1+y)}  \E^{-t} \bigl(u'(t)+u(t)-(u(0)-\tilde{c})\E^{t}\bigr)^2 dt, \\
  \int_{[x,y)}  d\tilde{\dip}(t) & =  \int_{[\log(1+x),\log(1+y))}  \E^{-t} d\dip(t),
\end{align*}
for all $x$, $y\in[0,\infty)$ with $x<y$ and constants $\tilde{c}\in\R$. 
\end{proof}

For applications to the conservative Camassa--Holm flow, it is also of interest to consider the isospectral problem for~\eqref{eqnCHISP} on the whole line.
Since the corresponding linear relation is a finite rank perturbation of two half-line problems (see the proof of~\cite[Lemma~5.2]{CHPencil} for example), the criteria from Theorem~\ref{thmCHdiscr} can readily be extended to the full line case.

\section{Schr\"odinger operators with \texorpdfstring{$\delta'$}{delta-prime}-interactions}\label{sec:Delta}

The results of Section~\ref{secDis} also apply to one-dimensional Schr\"odinger operators with $\delta'$-interactions. 
To simplify our considerations, we restrict ourselves to the case of the positive semi-axis. 
More specifically, let $\chi$ be a real-valued Borel measure on $[0,\infty)$ that is singular with respect to the Lebesgue measure. 
For the sake of simplicity, we shall also assume that $\chi$ does not have a point mass at zero. 
The Borel measure $\omega_\chi$ on $[0,\infty)$ is then defined as the sum of the measure $\chi$ and the Lebesgue measure, that is, 
\begin{align}\label{eq:omeganu}
\omega_\chi(B) = \chi(B) + \int_Bdx 
\end{align}
for every Borel set $B\subseteq [0,\infty)$. 
We consider the operator $H_\chi$ in the Hilbert space $L^2[0,\infty)$ associated with the differential expression
\begin{align}\label{eq:taunu}
\tau_\chi = -\frac{d}{dx}\frac{d}{d\omega_\chi(x)}
\end{align}
and subject to the Neumann boundary condition at zero. 
The operator $H_\chi$ can be viewed as a  {\em Hamiltonian with $\delta'$-interactions}. 
Namely, if $\chi$ is a discrete measure, that is,  
\begin{align}
\chi = \sum_{s\in X}\beta(s) \delta_s,
\end{align}
where $X$ is a discrete subset of $[0,\infty)$, $\beta$ is a real-valued function on $X$ and $\delta_s$ is the unit Dirac measure centred at $s$, then the differential expression $\tau_\chi$ can be formally written as (see \cite[Example 2.2]{dprime})
\begin{align}\label{eq:primeXdiscr}
 -\frac{d^2}{dx^2} + \sum_{s\in X}\beta(s)\langle \ledot,\delta_s'\rangle\delta_s',
\end{align}
which is the Hamiltonian with $\delta'$-interactions on $X$ of strength $\beta$ (see \cite{AGHH}, \cite{km13}, \cite{km14}).
It is known (see~\cite{dprime} and~\cite{MeasureSL}) that under the above assumptions on $\chi$, the operator $H_\chi$ is self-adjoint in $L^2[0,\infty)$. 
The spectral properties of $H_\chi$ turn out to be closely connected to the generalized indefinite string $S_\chi = (\infty,\omega_\chi,0)$; see \cite[Lemma~8.1]{ACSpec}.

\begin{lemma}\label{lemUnitEqv}
The operator $H_\chi$ is unitarily equivalent to the operator part of the linear relation $\T_\chi$ associated with the string $S_\chi$.
\end{lemma} 

Taking this connection into account and applying the results from Section~\ref{secDis}, we arrive at the following results for Hamiltonians with $\delta'$-interactions. 
As usual, we will denote with $\Qr$ the normalized anti-derivative of $\chi$, which is given by~\eqref{eqnAntiChi}.

\begin{theorem}\label{th:Dprime}
 The following assertions hold true:
\begin{enumerate}[label=(\roman*), ref=(\roman*), leftmargin=*, widest=iii]
\item\label{itmDprime1} 
Zero does not belong to the spectrum of $H_\chi$ if and only if there is a constant $c\in\R$ such that 
\begin{align}
\limsup_{x\to \infty}\, x\int_x^\infty (t + \Qr(t) - c)^2dt < \infty.
\end{align}
In this case, the constant $c$ is given by  
\begin{align}\label{eqnCprime}
c = \lim_{x\rightarrow \infty}\frac{x}{2} + \int_{[0,x)}\biggl(1-\frac{t}{x}\biggr)d\chi(t).
\end{align}
\item 
The spectrum of $H_\chi$ is  discrete if and only if there is a constant $c\in\R$ such that 
\begin{align}
\lim_{x\to \infty} x\int_x^\infty (t + \Qr(t) - c)^2dt =0.
\end{align}
\item 
For each $p>1$, the spectrum of $H_\chi$ satisfies
\begin{align}\label{eqnTraceHchi}
\sum_{\lambda\in\sigma(H_\chi)} \frac{1}{|\lambda|^p} < \infty
\end{align}
 if and only if there is a constant $c\in\R$ such that 
\begin{align}
 \int_0^\infty \biggl(x\int_x^\infty (t+ \Qr(t) - c)^2dt \biggr)^{\nicefrac{p}{2}} \frac{dx}{x} < \infty.
 \end{align}
 \item 
If  the spectrum of $H_\chi$ satisfies~\eqref{eqnTraceHchi} with $p=1$, then
\begin{align}
 \int_0^\infty \biggl(x\int_x^\infty (t+ \Qr(t) - c)^2dt \biggr)^{\nicefrac{1}{2}} \frac{dx}{x} < \infty,
 \end{align}
where the constant $c$ is given by~\eqref{eqnCprime}, and the function $x + \Qr - c$ is integrable with 
\begin{align}
\sum_{\lambda\in\sigma(H_\chi)} \frac{1}{\lambda} = \int_0^\infty (c - x - \Qr(x)) dx.
\end{align}
\end{enumerate}
\end{theorem}

\begin{proof}
Taking into account that the normalized anti-derivative of $\omega_\chi$ equals $\Wr_\chi(x)=x+ \Qr(x)$ for almost all $x\in [0,\infty)$, the result is a simple consequence of Lemma~\ref{lemUnitEqv} and Theorem~\ref{thmDSpec}. 
We only need to mention that the limit in \eqref{eqnCqrm} becomes
\[
\lim_{x\rightarrow\infty} \frac{1}{x}\int_0^{x} \Wr_\chi(t)dt = \lim_{x\rightarrow\infty} \frac{1}{x}\int_0^{x} t + \Qr(t) dt = \lim_{x\rightarrow \infty}\frac{x}{2} + \int_{[0,x)}\biggl(1-\frac{t}{x}\biggr)d\chi(t).\qedhere
\]
\end{proof}

Let us finish this section by applying Theorem \ref{th:Dprime} to the case of $\delta'$-interactions supported on a discrete set $X$ as in~\eqref{eq:primeXdiscr}. 
More specifically, suppose that 
\begin{align}\label{eqnchidis}
\chi = \sum_{k\in\N}\beta_k \delta_{x_k},
\end{align}
with $0=x_0<x_1<x_2<\dots $ and $x_k\uparrow \infty$ as $k\rightarrow \infty$. 
Clearly, in this case
\begin{align}
\Qr(x) = \int_{[0,x)}d\chi = \sum_{x_k < x} \beta_k = \sum_{k\in\N_0} \Qr_k\id_{[x_k,x_{k+1})}(x)
\end{align}
for almost all $x\in[0,\infty)$, where
\begin{align}
\Qr_0 &:= 0, & \Qr_k &:= \sum_{l\le k} \beta_l, \quad k\in\N.
\end{align}
The next result is an immediate consequence of Theorem~\ref{th:Dprime}~\ref{itmDprime1}.

\begin{corollary}\label{corBetaPrime}
Suppose that $\chi$ is of the form~\eqref{eqnchidis}. If the limit 
\begin{align}\label{eqnCbeta}
 \lim_{n\rightarrow \infty} \frac{n}{2} + \frac{1}{n}\sum_{x_k<n}(x_{k+1}-x_{k})\Qr_k
\end{align}
 does not exist or is infinite, then zero belongs to the essential spectrum of $H_\chi$.  
\end{corollary}

In fact, one can get a rather transparent characterization of discreteness.

\begin{corollary}\label{corBetaPrimeDiscr}
Suppose that $\chi$ is of the form~\eqref{eqnchidis}. 
The following assertions hold true: 
\begin{enumerate}[label=(\roman*), ref=(\roman*), leftmargin=*, widest=iii]
\item 
 Zero does not belong to the spectrum of $H_\chi$ if and only if 
 \begin{align}
\limsup_{n\to \infty}\, x_n\sum_{k\ge n} (x_{k+1} - x_k)^3 <\infty
 \end{align}
and, moreover, there is a constant $c\in\R$ such that 
 \begin{align}
\lim_{n\to \infty} x_n\sum_{k\ge n} (x_{k+1} - x_k)(\Qr_{k} + x_k - c)^2 <\infty.
 \end{align}
 In this case, the constant $c$ is given by the limit in~\eqref{eqnCbeta}.
\item 
The spectrum of $H_\chi$ is purely discrete if and only if 
 \begin{align}
\lim_{n\to \infty}\, x_n\sum_{k\ge n} (x_{k+1} - x_k)^3=0
 \end{align}
and, moreover, there is a constant $c\in\R$ such that 
 \begin{align}
\lim_{n\to \infty} x_n\sum_{k\ge n} (x_{k+1} - x_k)(\Qr_{k} + x_k - c)^2 = 0.
 \end{align}
\end{enumerate}
\end{corollary}

\begin{proof}
By Corollary~\ref{corBetaPrime}, we can assume that the limit in \eqref{eqnCbeta} exists and is finite.
We will denote it by $c_\chi$ and without loss of generality we can assume that it is zero. 
Then for all $x\in [x_{n},x_{n+1})$ and any $n\ge 0$ we get
\begin{align*}
 x\int_x^\infty (t + \Qr(t))^2dt &= x\int_x^{x_{n+1}} (t + \Qr_{n})^2dt + x \sum_{k> n}\int_{x_k}^{x_{k+1}} (t + \Qr_{k})^2dt\\
&= x(x_{n+1} - x)\biggl( \biggl(\Qr_{n}+\frac{x_{n+1}+x}{2}\biggr)^2 + \frac{(x_{n+1} - x)^2}{12}\biggr) \\
&\quad + x \sum_{k> n} (x_{k+1} - x_k)\biggl( \biggl(\Qr_{k}+\frac{x_{k+1}+x_{k}}{2}\biggr)^2 + \frac{(x_{k+1} - x_k)^2}{12}\biggr).
\end{align*}
Evaluating the above integral at $x_n$, we arrive at the estimate
\begin{align*}
x_{n}\int_{x_{n}}^\infty (t + \Qr(t))^2dt \ge \frac{x_{n}}{12}\sum_{k\ge n}(x_{k+1} - x_k)^3 + 
x_n\sum_{k\ge n} (x_{k+1} - x_k)(\Qr_{k} + x_k )^2
\end{align*}
for all $n\ge 1$, which together with Theorem~\ref{th:Dprime} immediately imply the necessity of the conditions in the claim. 

From the simple estimate 
\[
\biggl(\Qr_{k}+\frac{x_k+x_{k+1}}{2}\biggr)^2 \le 2 (\Qr_{k}+x_k)^2 + \frac{(x_{k+1} - x_k)^2}{2} 
\]
for all $k\ge 1$, we get the upper bound
\begin{align*}
 x\int_x^\infty (t + \Qr(t))^2dt & \le x_{n+1}(x_{n+1} - x_{n})\biggl( \big(\Qr_{n}+x_{n+1}\big)^2 + \frac{(x_{n+1} - x_{n})^2}{12}\biggr) \\
&\quad + x_{n+1} \sum_{k> n} (x_{k+1} - x_k)\Big( 2\big(\Qr_{k}+x_k\big)^2 + \frac{7}{12}(x_{k+1} - x_k)^2\Big) 
\end{align*}
for all $x\in [x_{n},x_{n+1})$ and any $n\ge 0$.
Taking into account that the first two summands are dominated by the last two, one easily proves sufficiency.
\end{proof}

\begin{remark}
Clearly, discreteness of the spectrum of $H_\chi$ is a rare event (for instance, by Theorem~\ref{th:Dprime}, existence of the limit in~\eqref{eqnCprime} is necessary; by employing a completely different approach, some sufficient conditions in the case when the support of $\chi$ is a discrete set were obtained in \cite[Section 6.4]{km10}). 
Moreover, in this case the eigenvalues of $H_\chi$ accumulate at $+\infty$ and at $-\infty$ (see \cite[Proposition~3.1]{km14}). 
\end{remark}

\appendix

\section{On a class of integral operators}\label{app:IntOp}

Let $\Qr$ be a function in $L^2_{\loc}[0,\infty)$ and consider the integral operator $\JIO$ in the Hilbert space $L^2[0,\infty)$ defined by 
\begin{align}\label{eq:a01}
\JIO f(x) = \int_0^\infty \Qr(\max(x,t))f(t)dt = \Qr(x)\int_0^x f(t)dt + \int_x^\infty \Qr(t)f(t)dt
\end{align}
for functions $f\in\dot{L}^2_{\cc}[0,\infty)$. 
Since the subspace $\dot{L}^2_{\cc}[0,\infty)$ is dense in $L^2[0,\infty)$, the operator $\JIO$ is densely defined. 
The theorems in this appendix gather a number of results from~\cite{AJPR} for these kinds of integral operators.

\begin{theorem}\label{thmAJPR}
The following assertions hold true:
\begin{enumerate}[label=(\roman*), ref=(\roman*), leftmargin=*, widest=iii]
\item\label{itmAJPRi} The operator $\JIO$ is bounded if and only if there is a constant $c\in\C$ such that
\begin{align}\label{eq:Jbndd}
\limsup_{x\rightarrow\infty}\, x\int_x^\infty |\Qr(t) - c|^2dt <\infty.
\end{align}
 In this case, the constant $c$ is given by 
\begin{align}\label{eqnDefc}
c = \lim_{x\rightarrow \infty} \frac{1}{x} \int_0^{x} \Qr(t)dt.
\end{align}
\item\label{itmAJPRii} The operator $\JIO$ is compact if and only if there is  a constant $c\in\C$ such that
\begin{align}\label{eq:Jcmct}
\lim_{x\rightarrow \infty} x\int_x^\infty |\Qr(t) - c|^2dt =0.
\end{align}
\item\label{itmAJPRiii} For each $p>1$, the operator $\JIO$ belongs to the Schatten--von Neumann class $\gS_p$ if and only if there is a constant $c\in\C$ such that 
\begin{align}\label{eq:Jsp}
 \int_0^\infty \biggl(x\int_x^\infty |\Qr(t) - c|^2dt\biggr)^{\nicefrac{p}{2}} \frac{dx}{x} < \infty.  
\end{align}
\item\label{itmAJPRiv} If the operator $\JIO$ belongs to the Hilbert--Schmidt class $\gS_2$, then its Hilbert--Schmidt norm is given by 
\begin{align}\label{eqnHSJ}
\|\JIO\|^2_{\gS_2}   =  2\int_0^\infty x |\Qr(x) - c|^2 dx,
\end{align}
where the constant $c$ is given by~\eqref{eqnDefc}.
\item\label{itmAJPRv} If the operator $\JIO$ belongs to the trace class $\gS_1$, then  
\begin{align}\label{eqnX1}
\int_0^\infty \biggl(x\int_x^\infty |\Qr(t) - c|^2dt\biggr)^{\nicefrac{1}{2}} \frac{dx}{x} <\infty,
\end{align}
  the function $\Qr-c$ is integrable and the trace of $\JIO$ is given by 
\begin{align}\label{eqnTraceJ}
\tr\,\JIO = \int_0^\infty (\Qr(x)-c) dx,
\end{align}
where the constant $c$ is given by~\eqref{eqnDefc}.
\end{enumerate}
\end{theorem}

\begin{proof}
Sufficiency of the conditions in \ref{itmAJPRi}, \ref{itmAJPRii} and \ref{itmAJPRiii} follows readily from \cite[Section~3]{AJPR} upon noticing that one has  
\[
  \JIO f(x) = \int_0^\infty \Qr(\max(x,t)) f(t)dt =  \int_0^\infty (\Qr(\max(x,t)) - c)f(t)dt 
\]
for functions $f \in \dot{L}^2_{\cc}[0,\infty)$. 
In order to prove that the condition in \ref{itmAJPRi} is also necessary, let us suppose that the operator $\JIO$ is bounded. 
For every $n\in\N$, we consider the function
\[
f_n = \id_{[0,1)} - \frac{1}{n}\id_{[1,n+1)}, 
\]
where $\id_I$ denotes the characteristic function of an interval $I\subseteq [0,\infty)$. 
Clearly, the functions $f_n$ belong to $\dot{L}^2_{\cc}[0,\infty)$ and converge  to $\id_{[0,1)}$ in $L^2[0,\infty)$ as $n\rightarrow \infty$. 
Since the operator $\JIO$ is bounded, this implies that the functions $\JIO f_n$ converge in $L^2[0,\infty)$. 
In view of the definition of $\JIO$, we then find that  
\begin{align*}
\JIO f_n(x) &=\begin{cases} 
x\Qr(x) + \int_x^1 \Qr(t)dt - \frac{1}{n}\int_1^{n+1}\Qr(t)dt, & x\in [0,1),\\[2mm]
\Qr(x) - \frac{1}{n}\Qr(x)(x-1) - \frac{1}{n}\int_x^{n+1}\Qr(t)dt, & x\in [1,n+1),\\[2mm]
0, & x\in[n+1,\infty), \end{cases}
\end{align*} 
for almost all $x\in[0,\infty)$. 
From this we are able to infer that the limit 
\begin{align*}
c := \lim_{n\rightarrow \infty} \frac{1}{n} \int_1^{n+1} \Qr(t)dt = \lim_{n\rightarrow \infty}\spr{Q - \JIO f_n}{\id_{[0,1)}}_{L^2[0,\infty)}
\end{align*}
exists in $\C$, where the function $Q$ in $L^2[0,\infty)$ is defined by 
\begin{align*}
Q(x) = \id_{[0,1)}(x)\biggl( x\Qr(x) + \int_x^1 \Qr(t)dt\biggr). 
\end{align*} 
Moreover, we see that for almost every $x\in[0,\infty)$ one has
\[
\lim_{n\rightarrow \infty} \JIO f_n(x) = \begin{cases} Q(x) - c, & x\in [0,1),\\ \Qr(x) - c, & x\in[1,\infty). \end{cases}
\]
Since, on the other side, it can be readily checked that one also has 
\[
 \int_0^\infty (\Qr(\max(x,t)) - c)\id_{[0,1)}(t)dt = \begin{cases} Q(x) - c, & x\in [0,1),\\ \Qr(x) - c, & x\in[1,\infty), \end{cases}
\]
we may conclude that the bounded extension of $\JIO$ to $L^2_\cc[0,\infty)$ satisfies 
\begin{align*}
\JIO \id_{[0,1)}(x) = \int_0^\infty (\Qr(\max(x,t)) - c)\id_{[0,1)}(t)dt
\end{align*}
and is hence  explicitly given by
\begin{align*}
\JIO f(x) = \int_0^\infty (\Qr(\max(x,t)) - c)f(t)dt
\end{align*}
for all functions $f \in L^2_{\cc}[0,\infty)$. 
 It now remains to apply~\cite[Theorem~3.1]{AJPR} to deduce that the function $\Qr$ satisfies \eqref{eq:Jbndd} and the constant $c$ is necessarily given by~\eqref{eqnDefc}. 
 If the operator $\JIO$ is moreover compact, then~\cite[Theorem~3.2]{AJPR} yields~\eqref{eq:Jcmct} and if it belongs to  $\gS_p$ for some $p>1$, then~\cite[Theorem~3.3]{AJPR} yields~\eqref{eq:Jsp}. 
 This proves that the conditions in \ref{itmAJPRii} and \ref{itmAJPRiii} are also necessary. 
 Finally, the formula~\eqref{eqnHSJ} for the Hilbert--Schmidt norm in~\ref{itmAJPRiv} follows from~\cite[Remark in Section~3]{AJPR} and the necessary condition~\eqref{eqnX1} for the operator $\JIO$ to belong to the trace class $\gS_1$ as well as the formula~\eqref{eqnTraceJ} for the trace in~\ref{itmAJPRv} follow from~\cite[Theorem~6.2]{AJPR}. 
\end{proof}

\begin{remark}
Let us stress that boundedness and compactness criteria for the integral operator $\JIO$ have been established before in \cite{chev70} and \cite{st73}, respectively.
\end{remark}

In Theorem~\ref{thmAJPR}~\ref{itmAJPRiii}, the value $p=1$ is a threshold since the condition~\eqref{eqnX1} is only necessary for the operator $\JIO$ to belong to the trace class $\gS_1$; see \cite[Section~6]{AJPR}.

\begin{theorem}\label{thmAJPR+}
Let $\chi$ be a non-negative Borel measure on $[0,\infty)$ and suppose that 
\begin{align}
  \Qr(x) = \int_{[0,x)}d\chi
 \end{align}
 for almost all $x\in[0,\infty)$.
 The following assertions hold true:
\begin{enumerate}[label=(\roman*), ref=(\roman*), leftmargin=*, widest=iii]
\item\label{itmAJPR+i} The operator $\JIO$ is bounded if and only if 
\begin{align}\label{eq:Jbndd+}
\limsup_{x\rightarrow\infty}\, x \int_{[x,\infty)}d\chi <\infty.
\end{align}
\item\label{itmAJPR+ii} The operator $\JIO$ is compact if and only if 
\begin{align}\label{eq:Jcmct+}
\lim_{x\rightarrow \infty} x \int_{[x,\infty)}d\chi  =0.
\end{align}
\item\label{itmAJPR+iii} For each $p>\nicefrac{1}{2}$, the operator $\JIO$ belongs to the Schatten--von Neumann class $\gS_p$ if and only if 
\begin{align}\label{eq:Jsp+}
\int_0^\infty  \biggl(x\int_{[x,\infty)}d\chi\biggr)^p\frac{dx}{x} <\infty. 
\end{align}
\item\label{itAJPR+iva} If the operator $\JIO$ belongs to the trace class $\gS_1$, then its trace is given by 
\begin{align}
\tr\,\JIO = -\int_{[0,\infty)} x\, d\chi(x).
\end{align}
\item\label{itAJPR+iv} If the operator $\JIO$ belongs to the Schatten--von Neumann class $\gS_{\nicefrac{1}{2}}$, then the measure $\chi$ is singular with respect to the Lebesgue measure.   
\end{enumerate}
\end{theorem}
 
\begin{proof}
By means of the connection established in the proof of Theorem~\ref{thmAJPR}, the claims in~\ref{itmAJPR+i}, \ref{itmAJPR+ii} and \ref{itmAJPR+iii} follow from~\cite[Theorem~4.6]{AJPR}, upon also noting that 
\[
  \lim_{x\rightarrow \infty}\frac{1}{x}\int_0^{x} \Qr(t)dt = \lim_{x\rightarrow \infty} \int_{[0,x)} \biggl(1-\frac{t}{x}\biggr) d\chi(t) = \int_{[0,\infty)} d\chi  
\] 
in this case. 
The claim in~\ref{itAJPR+iva} then follows from Theorem~\ref{thmAJPR}~\ref{itmAJPRv} and the claim in~\ref{itAJPR+iv} follows from~\cite[Corollary~8.12]{AJPR}.
\end{proof}

We also want to consider related operators on a finite interval. 
To this end, let $L$ be a positive number and define the operator $\JIO_L$ in the Hilbert space $L^2[0,L)$ by
\begin{align}\label{eq:a07}
 \JIO_L f(x) = \int_0^L \Qr_L(\min(x,t))f(t)dt = \int_0^x \Qr_L(t)f(t)dt + \Qr_L(x) \int_x^L f(t)dt
\end{align}
 for functions $f\in L^2_{\cc}[0,L)$, where $\Qr_L$ is a function in $L^2_{\loc}[0,L)$.
 As a Carleman integral operator, the operator $\JIO_L$ is closable (see \cite[Theorem~3.8]{hasu78} for example).

\begin{theorem}\label{thmAJPR2}
The following assertions hold true:
\begin{enumerate}[label=(\roman*), ref=(\roman*), leftmargin=*, widest=iii]
\item\label{itmAJPR2i} The operator $\JIO_L$ is bounded if and only if 
\begin{align}
 \limsup_{x\rightarrow L}\, (L-x)\int_0^x |\Qr_L(t)|^2dt < \infty.
\end{align}
\item\label{itmAJPR2ii} The operator $\JIO_L$ is compact if and only if 
\begin{align}
  \lim_{x\rightarrow L} (L-x)\int_0^x |\Qr_L(t)|^2dt = 0.
\end{align}
\item\label{itmAJPR2iii} For each $p>1$, the operator $\JIO_L$ belongs to the Schatten--von Neumann class $\gS_p$ if and only if  
\begin{align}
  \int_0^L \biggl((L-x)\int_0^x |\Qr_L(t)|^2dt\biggr)^{\nicefrac{p}{2}}\frac{dx}{L-x} < \infty.
\end{align}
\item\label{itmAJPR2iv} If the operator $\JIO_L$ belongs to the Hilbert--Schmidt class $\gS_2$, then its Hilbert--Schmidt norm is given by  
\begin{align}
\|\JIO_L\|^2_{\gS_2}   =  2 \int_0^L (L-x) |\Qr_L(x)|^2 dx.
\end{align}
\item\label{itmAJPR2v} If the operator $\JIO_L$ belongs to the trace class $\gS_1$, then    
\begin{align}\label{eqnX12}
   \int_0^L \biggl((L-x)\int_0^x |\Qr_L(t)|^2dt\biggr)^{\nicefrac{1}{2}}\frac{dx}{L-x} < \infty,
\end{align}
 the function $\Qr_L$ is integrable and the trace of $\JIO_L$ is given by  
\begin{align}\label{eqnTraceJL}
\tr\,\JIO_L = \int_0^L \Qr_L(x)dx.
\end{align}
\end{enumerate}
\end{theorem}

\begin{proof}
We first observe that for functions $f\in L^2_{\cc}[0,L)$ one has 
\begin{align*}
 \JIO_L f(L-x) & = \int_0^L \Qr_L(\min(L-x,t))f(t)dt \\
&= \int_0^L \Qr_L(\min(L-x,L-t))f(L-t)dt \\
&= \int_0^L \Qr_L(L - \max(x,t))f(L-t)dt 
\end{align*}
for almost all $x\in(0,L)$.  
 Taking into account that the map $f \mapsto f(L-\ledot)$ is unitary on $L^2[0,L)$, the claims in~\ref{itmAJPR2i}, \ref{itmAJPR2ii} and \ref{itmAJPR2iii} follow readily from the corresponding results in \cite[Section~3]{AJPR} with the function $\varphi$ in $L^2_{\loc}(0,\infty)$ given by 
\[
 \varphi(x)= \begin{cases} \Qr_L(L-x), & x\in (0,L), \\ 0, & x\in[L,\infty). \end{cases}
\]
The claims in~\ref{itmAJPR2iv} and \ref{itmAJPR2v} then follow from~\cite[Remark in Section~3]{AJPR} and~\cite[Theorem~6.2]{AJPR}, respectively.
\end{proof}

The value $p=1$ is again a threshold in Theorem~\ref{thmAJPR2}~\ref{itmAJPR2iii} because the condition~\eqref{eqnX12} is only necessary for the operator $\JIO_L$ to belong to the trace class $\gS_1$. 

\begin{theorem}\label{thmAJPR2+}
Let $\chi$ be a non-negative Borel measure on $[0,L)$ and suppose that 
\begin{align}
  \Qr_L(x) = \int_{[0,x)}d\chi
 \end{align}
 for almost all $x\in[0,L)$.
 The following assertions hold true:
\begin{enumerate}[label=(\roman*), ref=(\roman*), leftmargin=*, widest=iii]
\item\label{itmAJPR2+i} The operator $\JIO_L$ is bounded if and only if 
\begin{align}
\limsup_{x\rightarrow L}\, (L-x)\int_{[0,x)}d\chi <\infty.
\end{align}
\item\label{itmAJPR2+ii} The operator $\JIO_L$ is compact if and only if 
\begin{align}
\lim_{x\rightarrow L} (L-x)\int_{[0,x)}d\chi =0.
\end{align}
\item\label{itmAJPR2+iii} For each $p>\nicefrac{1}{2}$, the operator $\JIO_L$ belongs to the Schatten--von Neumann class $\gS_p$ if and only if 
\begin{align}
\int_0^L  \biggl((L-x)\int_{[0,x)}d\chi\biggr)^p\frac{dx}{L-x} <\infty. 
\end{align}
\item\label{itmAJPR2+iva} If the operator $\JIO_L$ belongs to the trace class $\gS_1$, then its trace is given by  
\begin{align}
\tr\,\JIO_L = \int_{[0,L)} (L-x) d\chi(x).
\end{align}
\item\label{itmAJPR2+iv} If the operator $\JIO_L$ belongs to the Schatten--von Neumann class $\gS_{\nicefrac{1}{2}}$, then the measure $\chi$ is singular with respect to the Lebesgue measure.  
\end{enumerate}
\end{theorem}

\begin{proof}
By means of the connection established in the proof of Theorem~\ref{thmAJPR2}, the claims in~\ref{itmAJPR2+i}, \ref{itmAJPR2+ii} and \ref{itmAJPR2+iii} follow from~\cite[Theorem~4.6]{AJPR}, the claim in~\ref{itmAJPR2+iva} follows from Theorem~\ref{thmAJPR2}~\ref{itmAJPR2v} and the claim in~\ref{itmAJPR2+iv} follows from~\cite[Corollary~8.12]{AJPR}.
\end{proof}

\section{Linear relations}\label{app:LinRel}

Let $\cH$ be a separable Hilbert space. A {\em (closed) linear relation} in $\cH$ is a (closed) linear subspace of $\cH\times\cH$. 
 Since every linear operator in $\cH$ can be identified with its graph, the set of linear operators can be regarded as a subset of all linear relations in $\cH$. 
Recall that the {\em domain}, the {\em range}, the {\em kernel} and the {\em multi-valued part} of a linear relation $\Theta$ are given, respectively, by
\begin{align}
\dom{\Theta} &= \{f\in \cH\,|\, \exists g\in\cH\ \text{such that}\ (f,g)\in \Theta\}, \\
\ran{\Theta} &= \{g\in \cH\,|\,  \exists f\in\cH\ \text{such that}\   (f,g)\in \Theta\},\\ 
\ker{\Theta} &= \{f\in \cH\,|\,  (f,0)\in \Theta\},\\
\mul{\Theta} &= \{g\in \cH\,|\,  (0,g)\in \Theta\}.
\end{align} 

The adjoint linear relation $\Theta^\ast$ of a linear relation $\Theta$ is defined by
\begin{align}
\Theta^\ast = \big\{ (\tilde{f},\tilde{g})\in \cH\times\cH\,|\,  \spr{g}{\tilde{f}}_{\cH} = \spr{f}{\tilde{g}}_{\cH}\ \text{for all}\ (f,g)\in\Theta\big\}.
\end{align} 
The linear relation $\Theta$ is called {\em symmetric} if $\Theta\subseteq \Theta^\ast$. It is called {\em self-adjoint} if $\Theta=\Theta^\ast$. Note that $\mul{\Theta}$ is orthogonal to $\dom{\Theta}$ if $\Theta$ is symmetric. For a closed symmetric linear relation $\Theta$ satisfying $\mul{\Theta} = \mul{\Theta^\ast}$ (the latter is further equivalent to the fact that $\Theta$ is densely defined on $\mul{\Theta}^\perp$), setting 
\begin{align}
\cH_{\rm op}=\overline{\dom{\Theta}} = \mul{\Theta}^\perp,
\end{align}
we obtain the following orthogonal decomposition
\begin{align}\label{eq:ThetaDecomp}
\Theta = \Theta_{\rm op}\oplus \Theta_{\infty},
\end{align}  
where $\Theta_\infty = \{0\}\times \mul{\Theta}$ and $\Theta_{\rm op}$ is the graph of a closed symmetric linear operator in $\cH_{\rm op}$, called the {\em operator part} of $\Theta$. 
Notice that for non-closed symmetric linear relations, the decomposition \eqref{eq:ThetaDecomp} may not hold true.

If $\Theta_1$ and $\Theta_2$ are linear relations in $\cH$, then their sum $\Theta_1+\Theta_2$ and their product $\Theta_2\Theta_1$ are defined by
\begin{align}
\Theta_1+\Theta_2 & = \{(f,g_1+g_2)\,|\, (f,g_1)\in\Theta_1,  \ (f,g_2)\in\Theta_2\},\\
\Theta_2\Theta_1 & = \{(f,g)\,|\, (f,h)\in\Theta_1,  \ (h,g)\in\Theta_2\ \text{for some}\ h\in\cH\}.
\end{align}
The inverse of a linear relation $\Theta$ is given by
\begin{align}
\Theta^{-1} = \{(g,f)\in \cH\times\cH \,|\, (f,g)\in \Theta\}.
\end{align}
Consequently, one can consider $(\Theta - z)^{-1}$ for any $z\in\C$. The set of those $z\in\C$ for which $(\Theta - z)^{-1}$ is the graph of a closed bounded operator on $\cH$ is called the {\em resolvent set} of $\Theta$ and denoted by $\rho(\Theta)$. Its complement 
$\sigma(\Theta)=\C\backslash\rho(\Theta)$ is called the {\em spectrum} of $\Theta$. If $\Theta$ is self-adjoint, then taking into account \eqref{eq:ThetaDecomp} we obtain
\begin{align}\label{eq:ThetaResolv}
(\Theta - z)^{-1} = (\Theta_{\rm op} - z)^{-1}\oplus \mathbb{O}_{\mul{\Theta}}.
\end{align}  
This immediately implies that $\rho(\Theta) = \rho(\Theta_{\rm op})$, $\sigma(\Theta) = \sigma(\Theta_{\rm op})$ and, moreover, one can introduce the spectral types of $\Theta$ as those of its operator part $\Theta_{\rm op}$.

\section{Estimates for Schatten--von Neumann norms}\label{app:UPD}

This section is an addendum to Appendix~\ref{app:IntOp}, aiming to provide explicit bounds on the sum in~\eqref{eqnSinSp}. In view of Corollary~\eqref{corPelin} and the connection between the operators $\KIO_\chi$ and the integral operators considered in Appendix~\ref{app:IntOp}, the problem reduces to the study of the operators $\JIO$. 
More specifically, let $\Qr$ be a function in $L^2_{\loc}[0,\infty)$ and consider the integral operator $\JIO$ defined by~\eqref{eq:a01} in the Hilbert space $L^2[0,\infty)$. We shall restrict our attention here to the case when $L=\infty$. We shall also assume in addition that $\Qr$ is real-valued (this will help to make the exposition more transparent and, moreover, it is exactly what is needed in our applications).

Clearly, one has $\JIO = \JIO_++\JIO_-$, where 
\begin{align}\label{eq:a02}
\JIO_+ f(x) & =  \Qr(x)\int_0^x f(t)dt, & \JIO_- f(x) & =  \int_x^\infty \Qr(t)f(t)dt,
\end{align}
for functions $f\in\dot{L}^2_{\cc}[0,\infty)$. 
Since the subspace $\dot{L}^2_{\cc}[0,\infty)$ is dense in $L^2[0,\infty)$, these operators are densely defined. It is known (see~\cite[Theorem~3.1 and Theorem~3.2]{AJPR} for example) that the operator $\JIO$ is bounded or compact in $L^2[0,\infty)$ if and only if so are $\JIO_+$ and $\JIO_-$. Moreover, $\JIO$ is compact if and only if $\Qr$ satisfies~\eqref{eq:Jcmct} with some constant $c \in\C$, which is then given by~\eqref{eqnDefc}.
In the following, we shall always assume~\eqref{eq:Jcmct} to be satisfied. We shall also assume that $c=0$ in~\eqref{eq:Jcmct} for transparence reasons (in this case, the closure of the operator $\JIO$, and hence of both $\JIO_+$ and $\JIO_-$, are given by the same expressions; therefore, we shall continue to use the same letters to denote their closures).
Notice that $\JIO_- = \JIO_+^\ast$ under the above assumptions on $\Qr$ so that $\JIO = 2\re\,\JIO_+$. In particular, this implies that $\JIO\in \gS_p$ whenever $\JIO_+\in \gS_p$ for $p\in (0,\infty]$. 
The converse is not necessarily true (in particular, it is not true for any $p\le 1$). However, by Matsaev's Theorem (see~\cite[Theorem~III.6.2]{gokr70} for example), for every $p\in (1,\infty)$ there is a constant $\gamma_p>0$ such that 
\begin{align}\label{eq:Matsaev}
\|\JIO_+\|_{\gS_p} \le \gamma_p \|\JIO\|_{\gS_p}.
\end{align}

  
\begin{remark}
The optimal constant $\gamma_p$ in~\eqref{eq:Matsaev} does not seem to be known.
However, it satisfies the following bounds (see~\cite[Eqs.~(III.6.11)--(III.6.12)]{gokr70}, which indeed provides better bounds)
\begin{align}
\gamma_p \le \frac{1}{2}\times \begin{cases} \frac{2p-1}{p-1}, & p\in (1,2),\\ 1+p, & p\in [2,\infty). \end{cases}
\end{align}
For further discussion of the constants $\gamma_p$ we refer to~\cite{gokr70}.
\end{remark}

Following the arguments in the proof of Theorem~3.3 in~\cite{AJPR}, one obtains the following result amplifying Theorem~\ref{thmAJPR}~\ref{itmAJPRiii}. 

\begin{lemma}\label{lem:AJPR}
Let $p>1$ and assume that $\Qr$ satisfies~\eqref{eq:Jsp} with $c=0$ so that the operator $\JIO$ belongs to the Schatten--von Neumann class $\gS_p$. Then there are positive constants $c_p$, $C_p>0$ such that 
\begin{align}\label{eq:JspConstants}
c_pE_p(\Qr)\le \|\JIO\|^p_{\gS_p} \le C_pE_p(\Qr),
\end{align}
where $E_p(\Qr)$ is the integral on the left-hand side in~\eqref{eq:Jsp}.
\end{lemma}

\begin{remark}
The constants in~\eqref{eq:JspConstants} can be made effective. Indeed, following the arguments in the proof of Theorem~3.3 in~\cite{AJPR} and using~\eqref{eq:Matsaev}, one gets the following estimate on the Schatten--von Neumann norm of the operator $\JIO$:
\begin{align}
2^{\nicefrac{-p}{2}}\tilde{E}_p(\Qr) \le \|\JIO\|^p_{\gS_p} \le 2^p \big(\gamma_p +\sqrt{2}+1 \big)^p\tilde{E}_p(\Qr),
\end{align}
where 
\begin{align}
\tilde{E}_p(\Qr)  = \sum_{n\in\Z} \Biggl(2^{n}\int_{2^{n}}^{2^{n+1}} |\Qr(x)|^2dx\Biggr)^{\nicefrac{p}{2}}.
\end{align}
On the other hand, it is an exercise to show that 
\begin{align}
2^{\nicefrac{-p}{2}-1} \tilde{E}_p(\Qr)\le E_p(\Qr) \le 2^{p} \tilde{E}_p(\Qr)\times
\begin{cases} \frac{1}{2^{\nicefrac{p}{2}}-1}, & p\in (0,2],\\[1mm] 1, & p\in (2,\infty).\end{cases}
\end{align}
\end{remark}

Let $\chi$ be a finite non-negative Borel measure on $[0,\infty)$ and suppose that the function $\Qr$ is given by
\begin{align}\label{eq:QrViaChi}
  \Qr(x) = \int_{[x,\infty)}d\chi
 \end{align}
 for almost all $x\in[0,\infty)$.
 Then, by Theorem~\ref{thmAJPR+},  for each $p>\nicefrac{1}{2}$, the operator $\JIO$ belongs to the Schatten--von Neumann class $\gS_p$ if and only if 
\eqref{eq:Jsp+} holds true. Again, one can estimate the Schatten--von Neumann norms in this case.

\begin{lemma}\label{lem:AJPR+}
Let $p>\nicefrac{1}{2}$ and let $\Qr$ be given by~\eqref{eq:QrViaChi} where $\chi$ is a non-negative Borel measure on $[0,\infty)$ satisfying~\eqref{eq:Jsp+}. Then there are positive constants $c_p^+$, $C_p^+>0$ such that 
\begin{align}\label{eq:JspConstants+}
c_p^+E_p^+(\chi)\le \|\JIO\|^p_{\gS_p} \le C_p^+E_p^+(\chi),
\end{align}
where $E_p^+(\chi)$ is the integral on the left-hand side in~\eqref{eq:Jsp+}.
\end{lemma}

For $p>1$ the above estimate follows from~\eqref{eq:JspConstants}. Indeed, denoting 
\begin{align}
 \tilde{E}_p^+(\chi) = 
   \sum_{n\in\Z} 2^{pn} \biggl(\int_{[2^n,\infty)} d\chi\biggr)^p,
\end{align}
one easily gets in view of monotonicity of $\Qr$ that 
\begin{align}
2^{-p} \tilde{E}_p^+(\chi) \le \tilde{E}_p(\Qr)  \le \tilde{E}_p^+(\chi).
\end{align}
It then only remains to notice the trivial bound
\begin{align}\label{eq:tiE+via+}
2^{-p-1} \tilde{E}_p^+(\chi) \le E_p^+(\chi) \le 2^p \tilde{E}_p^+(\chi).
\end{align}

The case $p\in(\nicefrac{1}{2},1]$ is based on the fact that the operator $\JIO$ admits the factorization
\begin{align}
\JIO = \tilde{\JIO}_+^\ast\tilde{\JIO}_+,
\end{align}
where $\tilde{\JIO}_+\colon L^2[0,\infty)\to L^2([0,\infty);\chi)$ is the integral operator given by
\begin{align}
\tilde{\JIO}_+f(x) = \int_0^x f(t)dt.
\end{align}
This clearly implies that $\JIO\in \gS_p$ exactly when $\tilde{\JIO}_+\in \gS_{2p}$ together with the equality $\|\JIO\|_{\gS_p} = \|\tilde{\JIO}_+\|_{\gS_{2p}}$. Thus, the corresponding two-sided estimate~\eqref{eq:JspConstants+} can be made effective for all $p>\nicefrac{1}{2}$ by following the proof of Theorem~4.4 in~\cite{AJPR}. More specifically, one gets
\begin{align}
2^{-p}\tilde{E}_p^+(\chi) \le \|\tilde{\JIO}_+\|^p_{\gS_{2p}} \le \big(\gamma_p +\sqrt{2}+1 \big)^p\tilde{E}_p^+(\chi),
\end{align}
and then it only remains to recall~\eqref{eq:tiE+via+}.

Finally, let us return to the study of the Schatten--von Neumann properties of generalized indefinite strings. For a given triple $(L,\omega,\dip)$ with $L=\infty$, as before we denote by $\sigma$ the spectrum of the corresponding self-adjoint linear relation $\T$. Then for each $p\ge 1$ we have the estimate
\begin{align}\label{eq:SpforTviaKs}
\frac{1}{2}\|\KIO_\omega\|_{\gS_p} + \frac{1}{2} \|\KIO_\dip\|_{\gS_{\nicefrac{p}{2}}} \le \biggl(\sum_{\lambda\in\sigma} \frac{1}{|\lambda|^p}\biggr)^{\nicefrac{1}{p}} \le \|\KIO_\omega\|_{\gS_p} + 2\|\KIO_\dip\|_{\gS_{\nicefrac{p}{2}}}.
\end{align}
Indeed, by Corollary~\eqref{corPelin}, $\lambda\in \sigma$ if and only if it belongs to the spectrum of the block-operator matrix~\eqref{eqnPeLinear}.
Therefore, the second inequality is an immediate consequence of the corresponding triangle inequality applied to the first unnumbered equation in the proof of Proposition~\ref{propDSpecM}. To prove the first inequality it suffices to notice that the Schatten--von Neumann norm of a block operator matrix is greater than the corresponding Schatten--von Neumann norm of any of its blocks and hence 
\begin{align}
\|\KIO_\omega\|_{\gS_p} & \le \|\T^{-1}\|_{\gS_p}, & \|\sqrt{\KIO}_\dip\|_{\gS_p} & \le \|\T^{-1}\|_{\gS_p}.
\end{align}  
The latter clearly implies the first inequality.
Let us also mention the trivial equality   
  \begin{align}
     \| \tilde{\T}\|_{\gS_p} = \| \KIO_\omega \|_{\gS_p},
   \end{align}
   which holds true for all $p>0$ whenever $\dip$ vanishes identically, which is always the case when $\sigma \subseteq [0,\infty)$. 

Finally, it suffices to mention that the operator $\KIO_\omega$ is unitarily equivalent to the operator $\JIO$ with the symbol $\Wr$ and that the operator $\KIO_\dip$ is unitarily equivalent to the operator $\JIO$ whose symbol is given by~\eqref{eq:QrViaChi} with $\dip$ in place of $\chi$. Therefore, the two-sided estimates in Lemma~\ref{lem:AJPR} and Lemma~\ref{lem:AJPR+} immediately imply the corresponding two sided estimates for the sum in~\eqref{eqnSinSp} by taking into account~\eqref{eq:SpforTviaKs}. The corresponding constants can be made effective and we leave this to the interested reader. Let us only mention the following estimate in the case of the positive spectrum. More specifically, if $(L,\omega,\dip)$ is a generalized indefinite string with $L=\infty$, $\dip$ vanishing identically and $\omega$ being a positive Borel measure on $[0,\infty)$, then  there are positive constants $c_+$, $C_+>0$ such that
\begin{align}\label{eq:SpforTviaW+}
c_+\int_0^\infty  \biggl(x\int_{[x,\infty)}d\omega\biggr)^p\frac{dx}{x} \le \sum_{\lambda\in\sigma} \frac{1}{|\lambda|^p} \le C_+ \int_0^\infty  \biggl(x\int_{[x,\infty)}d\omega\biggr)^p\frac{dx}{x}.
\end{align}

\begin{remark}
Consider the following functionals:
\begin{align}
E_p^A(\Qr) & = \sum_{n\in\Z} \left(2^{n}\int_{2^{n}}^{2^{n+1}} |\Qr(x)|^2dx\right)^{\nicefrac{p}{2}},\\
E_p^B(\Qr) & = \sum_{n\in\Z} \left(2^{n}\int_{2^{n}}^{\infty} |\Qr(x)|^2dx\right)^{\nicefrac{p}{2}},\\
E_p(\Qr) & = \int_0^\infty \left( x\int_x^{\infty} |\Qr(s)|^2ds\right)^{\nicefrac{p}{2}}\frac{dx}{x}.
\end{align}
Clearly,
\begin{align}
E_p^A(\Qr) \le E_p^B(\Qr).
\end{align}
On the other hand, 
\begin{align}\begin{split}
E_p^B(\Qr) & = \sum_{n\in\Z} 2^{np/2}\left(\int_{2^{n}}^{2^{n+1}}+\int_{2^{n+1}}^{\infty} |\Qr(x)|^2dx\right)^{\nicefrac{p}{2}} \\
& \le \begin{cases} E_p^A(\Qr) + 2^{\nicefrac{-p}{2}}E_p^B(\Qr), & p\in (0,2],\\ 2^{\nicefrac{p}{2}-1}E_p^A(\Qr) + 2^{-1} E_p^B(\Qr), & p\in (2,\infty).\end{cases}
\end{split}\end{align}
The first inequality follows from the trivial one
\begin{align}
(a+b)^q\le a^q+b^q,\qquad a,b\ge0,\quad q\in (0,1),
\end{align}
and the second estimate follows from Minkowski's inequality.
Therefore, we end up with the bound
\begin{align}
E_p^B(\Qr) \le E_p^A(\Qr)\times
\begin{cases} \frac{2^{\nicefrac{p}{2}}}{2^{\nicefrac{p}{2}}-1}, & p\in (0,2],\\[1mm] 2^{\nicefrac{p}{2}}, & p\in (2,\infty).\end{cases}
\end{align}
Finally, taking into account that
\begin{align}
\begin{split}
E_p(\Qr) & = \sum_{n\in\Z}\int_{2^{n}}^{2^{n+1}} \left( x\int_x^{\infty} |\Qr(s)|^2ds\right)^{\nicefrac{p}{2}}\frac{dx}{x}
\end{split} \\
\begin{split}
\text{estimate from above:  }& \le \sum_{n\in\Z} 2^{n}  \left( 2^{n+1} \int_{2^{n}}^{\infty} |\Qr(s)|^2ds\right)^{\nicefrac{p}{2}}\frac{dx}{2^{n}} \\
& \le 2^{p/2} E_p^B(\Qr)
\end{split} \\[2mm]
\begin{split}
\text{estimate from below:  } & \ge \sum_{n\in\Z} 2^{n}  \left( 2^{n} \int_{2^{n+1}}^{\infty} |\Qr(s)|^2ds\right)^{\nicefrac{p}{2}}\frac{dx}{2^{n+1}} \\
& \ge \sum_{n\in\Z} 2^{\nicefrac{-p}{2}}  \left( 2^{n+1} \int_{2^{n+1}}^{\infty} |\Qr(s)|^2ds\right)^{\nicefrac{p}{2}}\frac{dx}{2^{1}} \\
& = 2^{\nicefrac{-p}{2}-1} E_p^B(\Qr).
\end{split}
\end{align}
and using monotonicity, we end up with the two-sided estimate
\begin{align}
2^{-p/2-1} E_p^B(\Qr)\le E_p(\Qr) \le 2^{p/2} E_p^B(\Qr).
\end{align}

\end{remark}


\noindent
\section*{Acknowledgments}
We are grateful to Mark Malamud, Roman Romanov and Harald Woracek for useful discussions and hints with respect to the literature.


\end{document}